\documentclass{article}

\usepackage[english]{babel}
\usepackage{graphicx,mathdots,amsfonts,amssymb,amsmath,amsthm,mathrsfs,epstopdf,enumitem,color,url,booktabs}
\usepackage[top=2cm,bottom=3cm,left=2cm,right=2cm]{geometry}
\usepackage{sectsty}\sectionfont{\normalsize}\subsectionfont{\normalsize}

\newtheorem{theorem}{Theorem}[section]

\newtheorem{lemma}{Lemma}[section]

\theoremstyle{definition}

\newtheorem{remark}{Remark}[section]

\numberwithin{equation}{section}
\numberwithin{figure}{section}
\numberwithin{table}{section}

\newcommand{\e}{{\rm e}}
\renewcommand{\i}{{\rm i}}
\renewcommand{\epsilon}{\varepsilon}
\renewcommand{\phi}{\varphi}
\newcommand{\ee}{{\mathbf e}}
\newcommand{\ppp}{{\mathbf p}}
\newcommand{\xx}{{\mathbf x}}
\newcommand{\yy}{{\mathbf y}}
\newcommand{\uu}{{\mathbf u}}
\newcommand{\vv}{{\mathbf v}}
\newcommand{\ww}{{\mathbf w}}
\newcommand{\WW}{{\mathbf W}}
\newcommand{\XX}{{\mathbf X}}
\newcommand{\ii}{{\boldsymbol i}}
\newcommand{\jj}{{\boldsymbol j}}
\newcommand{\kk}{{\boldsymbol k}}
\newcommand{\pp}{{\boldsymbol p}}
\newcommand{\qq}{{\boldsymbol q}}
\newcommand{\nn}{{\boldsymbol n}}
\newcommand{\bsigma}{{\boldsymbol\sigma}}
\newcommand{\bmu}{{\boldsymbol\mu}}

\begin{document}

\title{Eigenvalues and Eigenvectors of Tau Matrices\\ with Applications to Markov Processes and Economics}

\author{Sven-Erik Ekstr\"om\\[-2pt]
\footnotesize Department of Information Technology, Division of Scientific Computing, Uppsala University, Sweden\\[-2pt]
\footnotesize Faculty of Mathematics and Natural Sciences, Bergische Universit\"at Wuppertal, Germany\\[7pt]
Carlo Garoni\thanks{Correspondence to: Carlo Garoni (garoni@mat.uniroma2.it).}\\[-2pt]
\footnotesize Department of Mathematics, University of Rome Tor Vergata, Italy\\[7pt]
Adam Jozefiak\\[-2pt]
\footnotesize Department of Computer Science, University of British Columbia, Canada\\[7pt]
Jesse Perla\\[-2pt]
\footnotesize Vancouver School of Economics, University of British Columbia, Canada}

\maketitle

\begin{abstract}
In the context of matrix displacement decomposition, Bozzo and Di Fiore introduced the so-called $\tau_{\epsilon,\phi}$ algebra, a generalization of the more known $\tau$ algebra originally proposed by Bini and Capovani. We study the properties of eigenvalues and eigenvectors of the generator $T_{n,\epsilon,\phi}$ of the $\tau_{\epsilon,\phi}$ algebra. In particular, we derive the asymptotics for the outliers of $T_{n,\epsilon,\phi}$ and the associated eigenvectors; we obtain equations for the eigenvalues of $T_{n,\epsilon,\phi}$, which provide also the eigenvectors of $T_{n,\epsilon,\phi}$; and we compute the full eigendecomposition of $T_{n,\epsilon,\phi}$ in the specific case $\epsilon\phi=1$. We also present applications of our results in the context of queuing models, random walks, and diffusion processes, with a special attention to their implications in the study of wealth/income inequality and portfolio dynamics.

\smallskip


\smallskip

\end{abstract}

\section{Introduction}\label{sec:intro}
Consider the $n\times n$ matrix 
\begin{equation*}\label{Tep}
T_{n,\epsilon,\phi}=\begin{bmatrix}
\epsilon & 1 & & & \\
1 & 0 & 1 & & \\
& \ddots & \ddots & \ddots & \\
& & 1 & 0 & 1\\
& & & 1 & \phi
\end{bmatrix},
\end{equation*}
where $\epsilon,\phi\in\mathbb R$ are given parameters.
For $\epsilon,\phi\in\{0,1,-1\}$, the eigendecomposition of $T_{n,\epsilon,\phi}$ is already available in the literature. In particular, for $(\epsilon,\phi)=(0,0)$, the matrix $T_{n,\epsilon,\phi}=T_{n,0,0}$ is the generator of the $\tau$ algebra originally introduced by Bini and Capovani~\cite{BC-tau}; its eigendecomposition, as well as the eigendecomposition of any tridiagonal Toeplitz matrix, has long been known \cite[Section~2.2]{BG}. For $(\epsilon,\phi)\ne(0,0)$, the matrix $T_{n,\epsilon,\phi}$ is the generator of the so-called $\tau_{\epsilon,\phi}$ algebra introduced by Bozzo and Di~Fiore in \cite{BozzoDiFiore}; its eigendecomposition for $(\epsilon,\phi)=(1,1)$, $(-1,-1)$, $(1,-1)$, $(-1,1)$ was provided in \cite[Section~4]{BozzoDiFiore}. Finally, for $(\epsilon,\phi)=(1,0)$, $(0,1)$, $(-1,0)$, $(0,-1)$---actually for all $\epsilon,\phi\in\{0,1,-1\}$---the eigendecomposition of $T_{n,\epsilon,\phi}$ can be derived, e.g., from the results in \cite[Appendix~1]{tau-appendix}; see in particular \cite[pp.~394--395]{tau-appendix}.

For all $\epsilon,\phi\in\mathbb R$, the asymptotic spectral distribution of $T_{n,\epsilon,\phi}$ in Weyl's sense can be easily obtained from the theory of generalized locally Toeplitz sequences \cite{GLTbookI,GLTbookII}, which immediately yields for $T_{n,\epsilon,\phi}$ the asymptotic spectral distribution function (or symbol) $2\cos\theta$. Precise eigenvalue estimates can also be given on the basis of classical interlacing results \cite[Section~4.3]{HJ} after observing that $T_{n,\epsilon,\phi}$ is a small-rank perturbation of $T_{n,0,0}$ and the eigenvalues of $T_{n,0,0}$ are known. It should be noted, however, that both asymptotic spectral distribution results and interlacing estimates completely ignore the outliers of $T_{n,\epsilon,\phi}$, i.e., the eigenvalues lying outside the interval $[-2,2]$ (the range of the symbol $2\cos\theta$). On the other hand, the outliers, which are determined by the parameters $\epsilon,\phi$, are precisely the objects one is interested in when dealing with several noteworthy applications. Such applications include, for example, queuing models and Markov chains/processes \cite{Baldi,BiniCIME,Giambene,lawler2006introduction}, where the eigenvector corresponding to the (unique) outlier of (a suitable transform of) $T_{n,\epsilon,\phi}$ corresponds to the steady-state distribution of the considered chain/process.

In this paper, we study the spectral properties of $T_{n,\epsilon,\phi}$ and present a few applications in the context of Markov chains/processes, with a special focus on queuing models, random walks, diffusion processes and economics issues.
The structure of the paper, including a summary of our contributions, is given below.
\begin{itemize}[nolistsep,leftmargin=*]
	\item In Section~\ref{sec:basic}, we study some basic spectral properties of $T_{n,\epsilon,\phi}$ that will simplify the analysis of later sections. 
	\item In Section~\ref{sec:aoto}, we derive the asymptotics of the outliers of $T_{n,\epsilon,\phi}$ and the associated eigenvectors. Our main results in this regard are Theorems~\ref{t_eps}--\ref{t_e=p}, which are validated through numerical experiments in Tables~\ref{T1}--\ref{T3}.
	\item In Section~\ref{eeT}, we derive equations for the eigenvalues of $T_{n,\epsilon,\phi}$. For all $\epsilon,\phi\in\mathbb R$ for which these equations can be solved, one obtains not only the eigenvalues but also the eigenvectors of $T_{n,\epsilon,\phi}$. Our main results in Section~\ref{eeT} are Theorems~\ref{p(-2,2)}--\ref{p=-2}. 
	\item In Section~\ref{sec:sep}, we solve the equations obtained in Section~\ref{eeT} for specific values of $\epsilon,\phi$. In particular, we show how it is possible to re-obtain through our equations the eigendecomposition of $T_{n,\epsilon,\phi}$ for $\epsilon,\phi\in\{0,1,-1\}$; and we address the new case $\epsilon\phi=1$, which is the case of interest for the applications presented in Section~\ref{sec:app}.
	\item In Section~\ref{sec:app}, we present a few applications in the context of Markov chains/processes, with a special focus on queuing models, random walks in a multidimensional lattice, multidimensional reflected diffusion processes and economics issues. In particular, we investigate the implications of our results within a model for wealth/income inequality and portfolio dynamics with an arbitrary number of assets: we provide analytical formulas for the steady-state (stationary) distribution of the underlying stochastic process (a multidimensional reflected diffusion process), we compute the convergence speed towards the steady state, and we also derive closed-form expressions for relevant moments of the stationary distribution such as the average wealth and the wealth variance.
	\item In Section~\ref{sec:conc}, we draw conclusions and outline possible future lines of research.
\end{itemize}

\section{Basic Properties of the Eigenvalues and Eigenvectors of $\boldsymbol{T_{n,\epsilon,\phi}}$}\label{sec:basic}

In this section, we collect some basic properties of the eigenvalues and eigenvectors of $T_{n,\epsilon,\phi}$ which will allow us to tackle the analysis of the next sections with useful a priori knowledge. Throughout this paper, the eigenvalues of $T_{n,\epsilon,\phi}$ which do not belong to the interval $[-2,2]$ are referred to as outliers. We denote by $\ee_1,\ldots,\ee_n$ the vectors of the canonical basis of $\mathbb R^n$, and by $E_n$ the symmetric permutation matrix (flip matrix) whose rows are those of the identity matrix $I_n$ in reverse order:
\[ E_n=\begin{bmatrix}& & 1\\ & \iddots & \\ 1 & & \end{bmatrix}. \]

\begin{theorem}\label{eig_props}
The following properties hold.
\begin{enumerate}[leftmargin=*,nolistsep]
	\item $T_{n,\phi,\epsilon}=E_nT_{n,\epsilon,\phi} E_n$.	It follows that $(\lambda,\uu)$ is an eigenpair of $T_{n,\epsilon,\phi}$ if and only if $(\lambda,E_n\uu)$ is an eigenpair of $T_{n,\phi,\epsilon}$.
	\item\label{vw} If $\epsilon\ne0$ and
	\[ \vv_n=[\epsilon^{-i+1}]_{i=1}^n=\begin{bmatrix}1\\ \epsilon^{-1}\\ \vdots\\ \epsilon^{-n+1}\end{bmatrix} \]
	then $T_{n,\epsilon,\phi}\vv_n-(\epsilon+\epsilon^{-1})\vv_n=\epsilon^{-n}(\epsilon\phi-1)\ee_n$. 
	Similarly, if $\phi\ne0$ and
	\[ \ww_n=[\phi^{-n+i}]_{i=1}^n=\begin{bmatrix}\phi^{-n+1}\\ \vdots\\ \phi^{-1}\\ 1\end{bmatrix} \]
	then $T_{n,\epsilon,\phi}\ww_n-(\phi+\phi^{-1})\ww_n=\phi^{-n}(\epsilon\phi-1)\ee_1$.
	\item\label{nondeg} $T_{n,\epsilon,\phi}$ has $n$ real distinct eigenvalues. 
	\item If $|\epsilon|,|\phi|\le1$ then all the eigenvalues of $T_{n,\epsilon,\phi}$ belong to $[-2,2]$.
	\item If $|\epsilon|\le1,\:|\phi|>1$ or $|\epsilon|>1,\:|\phi|\le1$ then all the eigenvalues of $T_{n,\epsilon,\phi}$ belong to $[-2,2]$ except for at most $1$ outlier.
	\item If $|\epsilon|,|\phi|>1$ then all the eigenvalues of $T_{n,\epsilon,\phi}$ belong to $[-2,2]$ except for at most $2$ outliers.
	\item If $|\epsilon|<1$ or $|\phi|<1$ then both $2$ and $-2$ are not eigenvalues of $T_{n,\epsilon,\phi}$.
\end{enumerate}
\end{theorem}
\begin{proof}
1. It follows from direct computation.

2. It follows from direct computation.

3. $T_{n,\epsilon,\phi}$ is nonderogatory just like any Hessenberg matrix with nonzero subdiagonal entries \cite[p.~82]{HJ}. Since $T_{n,\epsilon,\phi}$ is also real and symmetric (hence diagonalizable), we infer that $T_{n,\epsilon,\phi}$ has $n$ real distinct eigenvalues.

4. The result follows immediately from Gershgorin's theorem \cite[Theorem~6.1.1]{HJ}.

5. We prove the statement in the case where $|\epsilon|\le1$ and $|\phi|>1$ (the proof in the other case is identical). Write
\[ T_{n,\epsilon,\phi}=T_{n,\epsilon,0}+\phi\,\ee_n\ee_n^\top. \]
All the eigenvalues of $T_{n,\epsilon,0}$ belong to $[-2,2]$ by Gershgorin's theorem. Since the unique nonzero eigenvalue of the matrix $\phi\,\ee_n\ee_n^\top$ is $\phi$, it follows from a classical interlacing theorem \cite[Corollary~4.3.3]{HJ} that $n-1$ eigenvalues of $T_{n,\epsilon,\phi}$ belong to $[-2,2]$.

6.  Write
\[ T_{n,\epsilon,\phi}=T_{n,0,0}+\epsilon\,\ee_1\ee_1^\top+\phi\,\ee_n\ee_n^\top. \]
All the eigenvalues of $T_{n,0,0}$ belong to $[-2,2]$ by Gershgorin's theorem. Since the unique nonzero eigenvalues of the matrix $\epsilon\,\ee_1\ee_1^\top+\phi\,\ee_n\ee_n^\top$ are $\epsilon$ and $\phi$, it follows from \cite[Corollary~4.3.3]{HJ} that $n-2$ eigenvalues of $T_{n,\epsilon,\phi}$ belong to $[-2,2]$.

7. The result follows immediately from the fact that the matrix $T_{n,\epsilon,\phi}$ is irreducible and from the so-called Gershgorin's third theorem \cite[p.~80]{Bini}.
\end{proof}

\section{Asymptotics of the Outliers of $\boldsymbol{T_{n,\epsilon,\phi}}$}\label{sec:aoto}

If $|\epsilon|>1$ and $n$ is large enough, property~\ref{vw} of Theorem~\ref{eig_props} says that $(\epsilon+\epsilon^{-1},\vv_n)$ is substantially an eigenpair of $T_{n,\epsilon,\phi}$ (it is an exact eigenpair if $\epsilon\phi=1$). A similar consideration applies to $(\phi+\phi^{-1},\ww_n)$. The next theorems formalize this intuition. We remark that, for every $x>0$,
\[ x+x^{-1}=2\cosh(\log x)\ge2, \]
with equality holding if and only if $x=1$. In what follows, $\Lambda(X)$ denotes the spectrum of the matrix $X$.

\begin{lemma}\label{eig_conv}
The following properties hold.
\begin{enumerate}[leftmargin=*,nolistsep]
	\item If $|\epsilon|>1$ then there exists an eigenvalue $\mu_n$ of $T_{n,\epsilon,\phi}$ such that $\mu_n\to\epsilon+\epsilon^{-1}$ as $n\to\infty$. Since $\epsilon+\epsilon^{-1}>2$, the eigenvalue $\mu_n$ is eventually an outlier.
	\item If $|\phi|>1$ then there exists an eigenvalue $\nu_n$ of $T_{n,\epsilon,\phi}$ such that $\nu_n\to\phi+\phi^{-1}$ as $n\to\infty$. Since $\phi+\phi^{-1}>2$, the eigenvalue $\nu_n$ is eventually an outlier.
\end{enumerate}
\end{lemma}
\begin{proof}
1. Let $\{\uu_{1,n},\ldots,\uu_{n,n}\}$ be an orthonormal basis of $\mathbb R^n$ formed by eigenvectors of $T_{n,\epsilon,\phi}$ with corresponding eigenvalues $\lambda_{1,n},\ldots,\lambda_{n,n}$:
\[ T_{n,\epsilon,\phi}\uu_{i,n}=\lambda_{i,n}\uu_{i,n},\qquad i=1,\ldots,n. \]
Expand the vector $\vv_n=[1,\epsilon^{-1},\ldots,\epsilon^{-n+1}]^\top$ on this basis:
\begin{align}\label{0_eq}
\vv_n&=\sum_{i=1}^n\alpha_{i,n}\uu_{i,n},\\
\sum_{i=1}^n\alpha_{i,n}^2&=\|\vv_n\|_2^2=\frac{1-\epsilon^{-2n}}{1-\epsilon^{-2}}\to\frac1{1-\epsilon^{-2}}.\label{0_eq'}
\end{align}
The equation $T_{n,\epsilon,\phi}\vv_n-(\epsilon+\epsilon^{-1})\vv_n=\epsilon^{-n}(\epsilon\phi-1)\ee_n$ in Theorem~\ref{eig_props} becomes
\begin{equation}\label{a_eq}
\sum_{i=1}^n(\lambda_{i,n}-(\epsilon+\epsilon^{-1}))\alpha_{i,n}\uu_{i,n}=\epsilon^{-n}(\epsilon\phi-1)\ee_n.
\end{equation}
Passing to the norms, we obtain
\begin{equation}\label{b_eq}
\sum_{i=1}^n(\lambda_{i,n}-(\epsilon+\epsilon^{-1}))^2\alpha_{i,n}^2=\epsilon^{-2n}(\epsilon\phi-1)^2\to0.
\end{equation}
If we assume by contradiction that ${\rm dist}(\Lambda(T_{n,\epsilon,\phi}),\epsilon+\epsilon^{-1})=\min_{i=1,\ldots,n}|\lambda_{i,n}-(\epsilon+\epsilon^{-1})|\not\to0$ as $n\to\infty$, then there exists a positive constant $c$ such that 
\[ {\rm dist}(\Lambda(T_{n,\epsilon,\phi}),\epsilon+\epsilon^{-1})\ge c \]
frequently as $n\to\infty$, hence
\[ \sum_{i=1}^n(\lambda_{i,n}-(\epsilon+\epsilon^{-1}))^2\alpha_{i,n}^2\ge c^2\sum_{i=1}^n\alpha_{i,n}^2=c^2\|\vv_n\|_2^2\ge c^2 \]
frequently as $n\to\infty$, which is a contradiction to \eqref{b_eq}. We conclude that ${\rm dist}(\Lambda(T_{n,\epsilon,\phi}),\epsilon+\epsilon^{-1})\to0$ as $n\to\infty$, which is the thesis.

2. It follows from item~1 applied to $T_{n,\phi,\epsilon}$, taking into account that $\Lambda(T_{n,\phi,\epsilon})=\Lambda(T_{n,\epsilon,\phi})$ by Theorem~\ref{eig_props}.
\end{proof}

If $\xx,\yy\in\mathbb R^n$, we set $(\xx,\yy)=\xx^\top\yy$. If $\uu\in\mathbb R^n$, we denote by $P_{\uu}$ the orthogonal projector onto the subspace $\langle\uu\rangle$ generated by $\uu$. In the case where $\uu\ne\mathbf0$, the projector $P_\uu$ is explicitly given by
\[ P_\uu\xx=\frac{(\xx,\uu)}{(\uu,\uu)}\,\uu,\qquad\xx\in\mathbb R^n. \]

\begin{table}
\small
\centering
\caption{Validation of Theorem~\ref{t_eps} in the case $\epsilon=3$ and $\phi=1/2$ where $\epsilon+\epsilon^{-1}=3.\overline3$. For every $n$ we have denoted by $\mu_n$ the unique outlier of $T_{n,\epsilon,\phi}$ and by $\xx_n$ the corresponding normalized eigenvector computed by \textsc{Julia}.}
\label{T1}

\medskip

\begin{tabular}{rccclcl}
\toprule
$n$ && outlier $\mu_n$ && $|\mu_n-(\epsilon+\epsilon^{-1})|$ && $\|\xx_n-P_{\vv_n}\xx_n\|_2$\\
\midrule
8   && 3.3333333663723654 && $\hphantom{\mu_n}\,3.3\cdot10^{-8}$ && $\hphantom{\|\xx}\,3.0\cdot10^{-5}\vphantom{\int^{\Sigma^1}}$ \\
16  && 3.3333333333333341 && $\hphantom{\mu_n}\,7.7\cdot10^{-16}$ && $\hphantom{\|\xx}\,4.6\cdot10^{-9}\vphantom{\int^{\Sigma^1}}$ \\
32  && 3.3333333333333333 && $\hphantom{\mu_n}\,4.1\cdot10^{-31}$ && $\hphantom{\|\xx}\,1.1\cdot10^{-16}\vphantom{\int^{\Sigma^1}}$\\
64  && 3.3333333333333333 && $\hphantom{\mu_n}\,1.2\cdot10^{-61}$ && $\hphantom{\|\xx}\,5.8\cdot10^{-32}\vphantom{\int^{\Sigma^1}}$\\
128 && 3.3333333333333333 && $\hphantom{\mu_n}\,1.0\cdot10^{-122}$ && $\hphantom{\|\xx}\,1.7\cdot10^{-62}\vphantom{\int^{\Sigma^1}}$\\
\bottomrule
\end{tabular}
\end{table}

\begin{theorem}\label{t_eps}
Suppose that $|\epsilon|>1$ and $\phi\ne\epsilon$. Let $(\mu_n,\xx_n)$ be an eigenpair of $T_{n,\epsilon,\phi}$ such that $\mu_n\to\epsilon+\epsilon^{-1}$ as $n\to\infty$ and $\|\xx_n\|_2=1$ for all $n$.
Then, the following properties hold.
\begin{enumerate}[leftmargin=*,nolistsep]
	\item Eventually, $\mu_n$ is an outlier of $T_{n,\epsilon,\phi}$ and any other eigenvalue $\lambda_n\in\Lambda(T_{n,\epsilon,\phi})$ satisfies $|\lambda_n-(\epsilon+\epsilon^{-1})|\ge c$ for some positive constant $c$ independent of $n$.
	\item $\|\xx_n-P_{\vv_n}\xx_n\|_2\to0$ as $n\to\infty$, where $\vv_n=[1,\epsilon^{-1},\ldots,\epsilon^{-n+1}]^\top$.
\end{enumerate}
\end{theorem}
\begin{proof}
1. If $|\phi|\le1$ then all eigenvalues of $T_{n,\epsilon,\phi}$ belong to $[-2,2]$ except for at most 1 outlier (by Theorem~\ref{eig_props}). Since $\mu_n\to\epsilon+\epsilon^{-1}\not\in[-2,2]$, it is clear that $\mu_n$ coincides eventually with the unique outlier of $T_{n,\epsilon,\phi}$.
Moreover, any other eigenvalue $\lambda_n$ of $T_{n,\epsilon,\phi}$ satisfies the inequality $|\lambda_n-(\epsilon+\epsilon^{-1})|\ge c$ with 
\[ c={\rm dist}(\epsilon+\epsilon^{-1},[-2,2]). \] 
If $|\phi|>1$ then all eigenvalues of $T_{n,\epsilon,\phi}$ belong to $[-2,2]$ except for at most 2 outliers (by Theorem~\ref{eig_props}) and there exists an eigenvalue $\nu_n$ of $T_{n,\epsilon,\phi}$ such that $\nu_n\to\phi+\phi^{-1}\not\in[-2,2]$ (by Lemma~\ref{eig_conv}). Since $\mu_n\to\epsilon+\epsilon^{-1}\not\in[-2,2]$ and $\epsilon+\epsilon^{-1}\ne\phi+\phi^{-1}$ (because $\phi\ne\epsilon$ by assumption), it is clear that, eventually, $\mu_n\ne\nu_n$ and $\mu_n,\nu_n$ are the unique two outliers of $T_{n,\epsilon,\phi}$. Moreover, any eigenvalue $\lambda_n$ of $T_{n,\epsilon,\phi}$ with $\lambda_n\ne\mu_n$ satisfies eventually the inequality $|\lambda_n-(\epsilon+\epsilon^{-1})|\ge c$ with 
\[ c={\rm dist}(\epsilon+\epsilon^{-1},[-2,2]\cup[\phi+\phi^{-1}-\delta,\,\phi+\phi^{-1}+\delta]), \]
where $\delta$ is a fixed positive constant chosen so that $\epsilon+\epsilon^{-1}\not\in[\phi+\phi^{-1}-\delta,\,\phi+\phi^{-1}+\delta]$.

2. Let $\{\uu_{1,n},\ldots,\uu_{n,n}=\xx_n\}$ be an orthonormal basis of $\mathbb R^n$ formed by eigenvectors of $T_{n,\epsilon,\phi}$ with corresponding eigenvalues $\lambda_{1,n},\ldots,\lambda_{n,n}=\mu_n$:
\[ T_{n,\epsilon,\phi}\uu_{i,n}=\lambda_{i,n}\uu_{i,n},\qquad i=1,\ldots,n. \]
We expand the vector $\vv_n$ on this basis as in \eqref{0_eq} and we get \eqref{0_eq'}--\eqref{b_eq}. 
By item~1, we eventually have
\begin{equation}\label{again0}
\sum_{i=1}^n(\lambda_{i,n}-(\epsilon+\epsilon^{-1}))^2\alpha_{i,n}^2\ge c^2\sum_{i=1}^{n-1}\alpha_{i,n}^2+(\mu_n-(\epsilon+\epsilon^{-1}))^2\alpha_{n,n}^2.
\end{equation}
Hence, by \eqref{0_eq'} and \eqref{b_eq},
\begin{equation}
\sum_{i=1}^{n-1}\alpha_{i,n}^2\to0,\qquad\alpha_{n,n}^2\to\frac1{1-\epsilon^{-2}}.\label{1..n}
\end{equation}
Keeping in mind \eqref{0_eq}, \eqref{0_eq'} and \eqref{1..n}, we obtain
\begin{align}
\|\xx_n-P_{\vv_n}\xx_n\|_2^2&=\|\uu_{n,n}-P_{\vv_n}\uu_{n,n}\|_2^2=\left\|\uu_{n,n}-\frac{(\uu_{n,n},\vv_n)}{(\vv_n,\vv_n)}\,\vv_n\right\|_2^2=\left\|\uu_{n,n}-\frac{\alpha_{n,n}}{\|\vv_n\|_2^2}\sum_{i=1}^{n}\alpha_{i,n}\uu_{i,n}\right\|_2^2\notag\\
&=\left\|\biggl(1-\frac{\alpha_{n,n}^2}{\|\vv_n\|_2^2}\biggr)\uu_{n,n}+\frac{\alpha_{n,n}}{\|\vv_n\|_2^2}\sum_{i=1}^{n-1}\alpha_{i,n}\uu_{i,n}\right\|_2^2=\biggl(1-\frac{\alpha_{n,n}^2}{\|\vv_n\|_2^2}\biggr)^2+\frac{\alpha_{n,n}^2}{\|\vv_n\|_2^4}\sum_{i=1}^{n-1}\alpha_{i,n}^2\to0,\label{whathold}
\end{align}
which concludes the proof.
\end{proof}

The next theorem is completely analogous to Theorem~\ref{t_eps} and can be proved by the same type of argument or by using the relation between $T_{n,\epsilon,\phi}$ and $T_{n,\phi,\epsilon}$ (see Theorem~\ref{eig_props}).

\begin{theorem}\label{t_phi}
Suppose that $|\phi|>1$ and $\epsilon\ne\phi$. Let $(\nu_n,\yy_n)$ be an eigenpair of $T_{n,\epsilon,\phi}$ such that $\nu_n\to\phi+\phi^{-1}$ as $n\to\infty$ and $\|\yy_n\|_2=1$ for all $n$.
Then, the following properties hold.
\begin{enumerate}[leftmargin=*,nolistsep]
	\item Eventually, $\nu_n$ is an outlier of $T_{n,\epsilon,\phi}$ and any other eigenvalue $\lambda_n\in\Lambda(T_{n,\epsilon,\phi})$ satisfies $|\lambda_n-(\phi+\phi^{-1})|\ge c$ for some positive constant $c$ independent of $n$.
	\item $\|\yy_n-P_{\ww_n}\yy_n\|_2\to0$ as $n\to\infty$, where $\ww_n=[\phi^{-n+1},\ldots,\phi^{-1},1]^\top$.
\end{enumerate}
\end{theorem}

\begin{table}
\small
\centering
\caption{Validation of Theorems~\ref{t_eps} and~\ref{t_phi} in the case $\epsilon=4$ and $\phi=-2$ where $\epsilon+\epsilon^{-1}=4.25$ and $\phi+\phi^{-1}=-2.5$. For every $n$ we have denoted by $\mu_n,\nu_n$ the unique two outliers of $T_{n,\epsilon,\phi}$ and by $\xx_n,\yy_n$ the corresponding normalized eigenvectors computed by \textsc{Julia}. We have called $\mu_n$ the outlier closest to $\epsilon+\epsilon^{-1}$ and $\nu_n$ the other outlier.}
\label{T2}

\medskip

\begin{tabular}{rccclcl}
\toprule
$n$ && outlier $\mu_n$ && $\hspace{1.2pt}|\mu_n-(\epsilon+\epsilon^{-1})|$ && $\hspace{0.5pt}\|\xx_n-P_{\vv_n}\xx_n\|_2$\\[1pt]
&& outlier $\nu_n$ && $|\nu_n-(\phi+\phi^{-1})|$ && $\|\yy_n-P_{\ww_n}\yy_n\|_2$\\
\midrule
8   && $\hphantom{-}4.2499999950887285$ && $\hphantom{|\mu_n}\,4.9\cdot10^{-9}$ && $\hphantom{\|\xx}\,2.3\cdot10^{-5}\vphantom{\int^{\Sigma^1}}$ \\
    && $-2.4999484772090417$ && $\hphantom{|\mu_n}\,5.2\cdot10^{-5}$ && $\hphantom{\|\xx}\,5.9\cdot10^{-3}$\\[3pt]
16  && $\hphantom{-}4.2500000000000000$ && $\hphantom{|\mu_n}\,1.1\cdot10^{-18}$ && $\hphantom{\|\xx}\,3.5\cdot10^{-10}\vphantom{\int^{\Sigma^1}}$\\
    && $-2.4999999992141966$ && $\hphantom{|\mu_n}\,7.9\cdot10^{-10}$  && $\hphantom{\|\xx}\,2.3\cdot10^{-5}$\\[3pt]
32  && $\hphantom{-}4.2500000000000000$    && $\hphantom{|\mu_n}\,6.2\cdot10^{-38}$  && $\hphantom{\|\xx}\,8.1\cdot10^{-20}\vphantom{\int^{\Sigma^1}}$ \\
    && $-2.5000000000000000$ && $\hphantom{|\mu_n}\,1.8\cdot10^{-19}$  && $\hphantom{\|\xx}\,3.5\cdot10^{-10}$\\[3pt]
64  && $\hphantom{-}4.2500000000000000$    && $\hphantom{|\mu_n}\,1.8\cdot10^{-76}$  && $\hphantom{\|\xx}\,4.4\cdot10^{-39}\vphantom{\int^{\Sigma^1}}$ \\
    && $-2.5000000000000000$ && $\hphantom{|\mu_n}\,9.9\cdot10^{-39}$  && $\hphantom{\|\xx}\,8.1\cdot10^{-20}$\\[3pt]
128 && $\hphantom{-}4.2500000000000000$    && $\hphantom{|\mu_n}\,1.6\cdot10^{-153}$ && $\hphantom{\|\xx}\,1.3\cdot10^{-77}\vphantom{\int^{\Sigma^1}}$ \\
    && $-2.5000000000000000$ && $\hphantom{|\mu_n}\,2.9\cdot10^{-77}$  && $\hphantom{\|\xx}\,4.4\cdot10^{-39}$\\
\bottomrule
\end{tabular}
\end{table}

To conclude our analysis, we address the case where $|\epsilon|,|\phi|>1$ and $\epsilon=\phi$.

\begin{theorem}\label{t_e=p}
Suppose that $|\epsilon|,|\phi|>1$ and $\epsilon=\phi$. Then, the following properties hold.
\begin{enumerate}[leftmargin=*,nolistsep]
	\item There exist exactly two distinct eigenvalues $\mu_n,\nu_n$ of $T_{n,\epsilon,\phi}$ which are eventually the unique two outliers of $T_{n,\epsilon,\phi}$ and satisfy $\mu_n,\nu_n\to\epsilon+\epsilon^{-1}=\phi+\phi^{-1}$.
	\item Let $\xx_n$ and $\yy_n$ be eigenvectors of $T_{n,\epsilon,\phi}$ associated with $\mu_n$ and $\nu_n$, respectively, and satisfying $\|\xx_n\|_2=\|\yy_n\|_2=1$ for all $n$. Then, up to a renaming of $\mu_n$ and $\nu_n$, we eventually have $E_n\xx_n=\xx_n$ and $E_n\yy_n=-\yy_n$. Moreover, $\|\xx_n-P_{\vv_n+\ww_n}\xx_n\|_2\to0$ and $\|\yy_n-P_{\vv_n-\ww_n}\yy_n\|_2\to0$ as $n\to\infty$, where $\vv_n=[1,\epsilon^{-1},\ldots,\epsilon^{-n+1}]^\top$ and $\ww_n=[\phi^{-n+1},\ldots,\phi^{-1},1]^\top=E_n\vv_n$.
\end{enumerate}
\end{theorem}
\begin{proof}
1. We first recall that all eigenvalues of $T_{n,\epsilon,\phi}$ are distinct by Theorem~\ref{eig_props}. Also, an eigenvalue converging to $\epsilon+\epsilon^{-1}$ exists for sure by Lemma~\ref{eig_conv} and more than two eigenvalues converging to $\epsilon+\epsilon^{-1}$ cannot exist by Theorem~\ref{eig_props} as $\epsilon+\epsilon^{-1}\not\in[-2,2]$. Suppose by contradiction that there exists a unique eigenvalue $\mu_n$ converging to $\epsilon+\epsilon^{-1}$ and let $\xx_n$ be a corresponding eigenvector with $\|\xx_n\|_2=1$.
Let $\{\uu_{1,n},\ldots,\uu_{n,n}=\xx_n\}$ be an orthonormal basis of $\mathbb R^n$ formed by eigenvectors of $T_{n,\epsilon,\phi}$ with corresponding eigenvalues $\lambda_{1,n},\ldots,\lambda_{n,n}=\mu_n$:
\[ T_{n,\epsilon,\phi}\uu_{i,n}=\lambda_{i,n}\uu_{i,n},\qquad i=1,\ldots,n. \]
We expand the vector $\vv_n$ on this basis as in \eqref{0_eq} and we get \eqref{0_eq'}--\eqref{b_eq}. Since $\mu_n$ is the unique eigenvalue of $T_{n,\epsilon,\phi}$ converging to $\epsilon+\epsilon^{-1}\notin[-2,2]$ and $n-2$ eigenvalues of $T_{n,\epsilon,\phi}$ belong to $[-2,2]$ for all $n$, there exists a positive constant $c$ independent of $n$ such that
\begin{equation}\label{again}
\sum_{i=1}^n(\lambda_{i,n}-(\epsilon+\epsilon^{-1}))^2\alpha_{i,n}^2\ge c^2\sum_{i=1}^{n-1}\alpha_{i,n}^2+(\mu_n-(\epsilon+\epsilon^{-1}))^2\alpha_{n,n}^2
\end{equation}
frequently as $n\to\infty$. Passing to a subsequence of indices $n$, if necessary, we may assume that \eqref{again} is satisfied for all $n$. Note that \eqref{again} is the same as \eqref{again0}. Hence, by reasoning as before, we infer that \eqref{1..n}--\eqref{whathold} hold and we conclude that $\|\xx_n-P_{\vv_n}\xx_n\|_2\to0$ (for the considered subsequence of indices $n$). This is impossible for the following reasons.
\begin{itemize}[leftmargin=*,nolistsep]
	\item Since $\epsilon=\phi$, we have $T_{n,\epsilon,\phi}=T_{n,\phi,\epsilon}$ and, by Theorem~\ref{eig_props}, $(\lambda,\uu)$ is an eigenpair of $T_{n,\epsilon,\phi}$ if and only if the same is true for $(\lambda,E_n\uu)$.
	\item By Theorem~\ref{eig_props}, each eigenvalue $\lambda$ of $T_{n,\epsilon,\phi}$ is simple and so $E_n\uu=\pm\uu$ for all eigenvectors $\uu$ of $T_{n,\epsilon,\phi}$. In particular $E_n\xx_n=\pm\xx_n$ for all $n$.
	\item If $\|\xx_n-P_{\vv_n}\xx_n\|_2\to0$ then the relation $E_n\xx_n=\pm\xx_n$ cannot hold for all $n$. Indeed, considering that $P_{\vv_n}\xx_n=c_n\vv_n$ is a multiple of $\vv_n$, from $\|\xx_n-P_{\vv_n}\xx_n\|_2\to0$ and $\|\xx_n\|_2=1$ we deduce that $\|P_{\vv_n}\xx_n\|_2=|c_n|\,\|\vv_n\|_2\to1$, i.e., $c_n\to1-\epsilon^{-2}$ (see \eqref{0_eq'}), and 
	\begin{align*}
	|(\xx_n)_1-(P_{\vv_n}\xx_n)_1|&=|(\xx_n)_1-c_n|\to0,\\
	|(\xx_n)_n-(P_{\vv_n}\xx_n)_n|&=|(\xx_n)_n-c_n\epsilon^{-n+1}|\to0,
	\end{align*}
	which are clearly incompatible with $E_n\xx_n=\pm\xx_n$ as the latter implies $(\xx_n)_n=\pm(\xx_n)_1$.
\end{itemize}

2. Let $\{\uu_{1,n},\ldots,\uu_{n-1,n}=\xx_n,\uu_{n,n}=\yy_n\}$ be an orthonormal basis of $\mathbb R^n$ formed by eigenvectors of $T_{n,\epsilon,\phi}$ with corresponding eigenvalues $\lambda_{1,n},\ldots,\lambda_{n-1,n}=\mu_n,\lambda_{n,n}=\nu_n$:
\[ T_{n,\epsilon,\phi}\uu_{i,n}=\lambda_{i,n}\uu_{i,n},\qquad i=1,\ldots,n. \]
Expand the vectors $\vv_n+\ww_n$ and $\vv_n-\ww_n$ on this basis:
\begin{align}\label{0_eq+}
\vv_n+\ww_n&=\sum_{i=1}^n\rho_{i,n}\uu_{i,n},\\
\vv_n-\ww_n&=\sum_{i=1}^n\tau_{i,n}\uu_{i,n},\label{0_eq++}\\
\sum_{i=1}^n\rho_{i,n}^2&=\|\vv_n+\ww_n\|_2^2=2\,\frac{1-\epsilon^{-2n}}{1-\epsilon^{-2}}+2n\epsilon^{-n+1}\to\frac2{1-\epsilon^{-2}},\label{0_eq'+}\\
\sum_{i=1}^n\tau_{i,n}^2&=\|\vv_n-\ww_n\|_2^2=2\,\frac{1-\epsilon^{-2n}}{1-\epsilon^{-2}}-2n\epsilon^{-n+1}\to\frac2{1-\epsilon^{-2}}.\label{0_eq'++}
\end{align}
Keeping in mind that $\epsilon=\phi$, the equations 
\begin{align*}
T_{n,\epsilon,\phi}\vv_n-(\epsilon+\epsilon^{-1})\vv_n&=\epsilon^{-n}(\epsilon\phi-1)\ee_n,\\
T_{n,\epsilon,\phi}\ww_n-(\phi+\phi^{-1})\ww_n&=\phi^{-n}(\epsilon\phi-1)\ee_1
\end{align*}
in Theorem~\ref{eig_props} yield
\begin{align*}
T_{n,\epsilon,\phi}(\vv_n+\ww_n)-(\epsilon+\epsilon^{-1})(\vv_n+\ww_n)&=\epsilon^{-n}(\epsilon\phi-1)(\ee_n+\ee_1),\\
T_{n,\epsilon,\phi}(\vv_n-\ww_n)-(\epsilon+\epsilon^{-1})(\vv_n-\ww_n)&=\epsilon^{-n}(\epsilon\phi-1)(\ee_n-\ee_1),
\end{align*}
that is,
\begin{align*}
\sum_{i=1}^n(\lambda_{i,n}-(\epsilon+\epsilon^{-1}))\rho_{i,n}\uu_{i,n}=\epsilon^{-n}(\epsilon\phi-1)(\ee_n+\ee_1),\\
\sum_{i=1}^n(\lambda_{i,n}-(\epsilon+\epsilon^{-1}))\tau_{i,n}\uu_{i,n}=\epsilon^{-n}(\epsilon\phi-1)(\ee_n-\ee_1).
\end{align*}
Passing to the norms, we obtain
\begin{align}\label{b_eq+}
\sum_{i=1}^n(\lambda_{i,n}-(\epsilon+\epsilon^{-1}))^2\rho_{i,n}^2=2\epsilon^{-2n}(\epsilon\phi-1)^2\to0,\\
\sum_{i=1}^n(\lambda_{i,n}-(\epsilon+\epsilon^{-1}))^2\tau_{i,n}^2=2\epsilon^{-2n}(\epsilon\phi-1)^2\to0.\label{b_eq++}
\end{align}
Since $\mu_n,\nu_n$ are eventually the unique two outliers of $T_{n,\epsilon,\phi}$, the other $n-2$ eigenvalues $\lambda_{1,n},\ldots,\lambda_{n-2,n}$ eventually belong to $[-2,2]$ and from \eqref{0_eq'+}--\eqref{b_eq++} we infer that
\begin{align}
\sum_{i=1}^{n-2}\rho_{i,n}^2\to0,\qquad\rho_{n-1,n}^2+\rho_{n,n}^2\to\frac2{1-\epsilon^{-2}},\label{1..n+}\\
\sum_{i=1}^{n-2}\tau_{i,n}^2\to0,\qquad\tau_{n-1,1}^2+\tau_{n,n}^2\to\frac2{1-\epsilon^{-2}}.\label{1..n++}
\end{align}
Now, recall from the proof of item~1 that (in the present case where $\epsilon=\phi$) all eigenvectors $\uu$ of $T_{n,\epsilon,\phi}$ satisfy $E_n\uu=\pm\uu$. Since $E_n(\vv_n+\ww_n)=\vv_n+\ww_n$ and $E(\vv_n-\ww_n)=-(\vv_n-\ww_n)$, for the eigenvectors $\uu_{i,n}$ satisfying $E_n\uu_{i,n}=\uu_{i,n}$ we have $\tau_{i,n}=0$ in the expansion \eqref{0_eq++}, and for the eigenvectors $\uu_{i,n}$ satisfying $E_n\uu_{i,n}=-\uu_{i,n}$ we have $\rho_{i,n}=0$ in the expansion \eqref{0_eq+}. It follows that, eventually, one among $\xx_n$ and $\yy_n$ (say $\xx_n$) must satisfy $E_n\xx_n=\xx_n$ and the other (say $\yy_n$) must satisfy the ``opposite'' equation $E_n\yy_n=-\yy_n$. Indeed, if we frequently had $E_n\xx_n=\xx_n$ and $E_n\yy_n=\yy_n$, then we would also have $\tau_{n-1,n}=\tau_{n,n}=0$ frequently, which is impossible by \eqref{1..n++}. Similarly, if we frequently had $E_n\xx_n=-\xx_n$ and $E_n\yy_n=-\yy_n$, then we would also have $\rho_{n-1,n}=\rho_{n,n}=0$ frequently, which is impossible by \eqref{1..n+}. By renaming $\mu_n$ and $\nu_n$ (if necessary), we can assume that the eigenvector $\xx_n$ associated with $\mu_n$ eventually satisfies $E_n\xx_n=\xx_n$, and the eigenvector $\yy_n$ associated with $\nu_n$ eventually satisfies $E_n\yy_n=-\yy_n$. In particular, we eventually have
\begin{align}
\rho_{n,n}&=0,\label{rho0}\\
\tau_{n-1,n}&=0.\label{tau0}
\end{align}
Thus, by applying \eqref{0_eq+}, \eqref{0_eq'+}, \eqref{1..n+} and \eqref{rho0}, we eventually obtain
\begin{align*}
\|\xx_n-P_{\vv_n+\ww_n}\xx_n\|_2^2&=\|\uu_{n-1,n}-P_{\vv_n+\ww_n}\uu_{n-1,n}\|_2^2=\left\|\uu_{n-1,n}-\frac{(\uu_{n-1,n},\vv_n+\ww_n)}{(\vv_n+\ww_n,\vv_n+\ww_n)}(\vv_n+\ww_n)\right\|_2^2\\
&=\left\|\uu_{n-1,n}-\frac{\rho_{n-1,n}}{\|\vv_n+\ww_n\|_2^2}\sum_{i=1}^n\rho_{i,n}\uu_{i,n}\right\|_2^2\\
&=\left\|\biggl(1-\frac{\rho_{n-1,n}^2}{\|\vv_n+\ww_n\|_2^2}\biggr)\uu_{n-1,n}-\frac{\rho_{n-1,n}}{\|\vv_n+\ww_n\|_2^2}\sum_{i=1}^{n-2}\rho_{i,n}\uu_{i,n}\right\|_2^2\\
&=\biggl(1-\frac{\rho_{n-1,n}^2}{\|\vv_n+\ww_n\|_2^2}\biggr)^2+\frac{\rho_{n-1,n}^2}{\|\vv_n+\ww_n\|_2^4}\sum_{i=1}^{n-2}\rho_{i,n}^2\to0.
\end{align*}
Similarly, one can show that $\|\yy_n-P_{\vv_n-\ww_n}\yy_n\|_2^2\to0$.
\end{proof}

\begin{table}
\small
\centering
\caption{Validation of Theorem~\ref{t_e=p} in the case $\epsilon=\phi=8/5$ where $\epsilon+\epsilon^{-1}=\phi+\phi^{-1}=2.225$. For every $n$ we have denoted by $\mu_n,\nu_n$ the unique two outliers of $T_{n,\epsilon,\phi}$ and by $\xx_n,\yy_n$ the corresponding normalized eigenvectors computed by \textsc{Julia}. We have called $\mu_n$ the outlier whose eigenvector $\xx_n$ is the closest to its projection onto $\langle\vv_n+\ww_n\rangle$ and $\nu_n$ the other outlier. We have numerically verified that, up to rounding errors, $E_n\xx_n=\xx_n$ and $E_n\yy_n=-\yy_n$ for all the considered $n$.}
\label{T3}

\medskip

\begin{tabular}{rccclcl}
\toprule
$n$ && outlier $\mu_n$ && $\hspace{1.2pt}|\mu_n-(\epsilon+\epsilon^{-1})|$ && $\|\xx_n-P_{\vv_n+\ww_n}\xx_n\|_2$\\[1pt] 
&& outlier $\nu_n$ && $|\nu_n-(\phi+\phi^{-1})|$ && $\|\yy_n-P_{\vv_n-\ww_n}\yy_n\|_2$\\
\midrule
8   && 2.2447548446486838 && $\hphantom{|\mu_n}\,2.0\cdot10^{-2}$ && $\hphantom{\|\xx_n\!-}2.1\cdot10^{-2}\vphantom{\int^{\Sigma^1}}$\\
    && 2.1991364375014231 && $\hphantom{|\mu_n}\,2.6\cdot10^{-2}$ && $\hphantom{\|\xx_n\!-}1.2\cdot10^{-2}$\\[3pt]
16  && 2.2255116405185864 && $\hphantom{|\mu_n}\,5.1\cdot10^{-4}$ && $\hphantom{\|\xx_n\!-}5.9\cdot10^{-4}\vphantom{\int^{\Sigma^1}}$\\
    && 2.2244808853312168 && $\hphantom{|\mu_n}\,5.2\cdot10^{-5}$ && $\hphantom{\|\xx_n\!-}4.9\cdot10^{-4}$\\[3pt]
32  && 2.2250002793612006 && $\hphantom{|\mu_n}\,2.8\cdot10^{-7}$ && $\hphantom{\|\xx_n\!-}2.9\cdot10^{-7}\vphantom{\int^{\Sigma^1}}$\\
    && 2.2249997206340419 && $\hphantom{|\mu_n}\,2.8\cdot10^{-7}$ && $\hphantom{\|\xx_n\!-}2.9\cdot10^{-7}$\\[3pt]
64  && 2.2250000000000821 && $\hphantom{|\mu_n}\,8.2\cdot10^{-14}$ && $\hphantom{\|\xx_n\!-}8.6\cdot10^{-14}\vphantom{\int^{\Sigma^1}}$\\
    && 2.2249999999999180 && $\hphantom{|\mu_n}\,8.2\cdot10^{-14}$ && $\hphantom{\|\xx_n\!-}8.6\cdot10^{-14}$\\[3pt]
128 && 2.2250000000000000 && $\hphantom{|\mu_n}\,7.1\cdot10^{-27}$ && $\hphantom{\|\xx_n\!-}7.5\cdot10^{-27}\vphantom{\int^{\Sigma^1}}$\\
    && 2.2250000000000000 && $\hphantom{|\mu_n}\,7.1\cdot10^{-27}$ && $\hphantom{\|\xx_n\!-}7.5\cdot10^{-27}$\\
\bottomrule
\end{tabular}
\end{table}

In Tables~\ref{T1}--\ref{T3}, we validate through numerical experiments the results presented in Theorems~\ref{t_eps}--\ref{t_e=p}. The experiments have been performed via the high-performance computing language \textsc{Julia} \cite{Julia} with a machine precision equal to $1.1\cdot10^{-308}$ (1024-bit precision).
We note that the convergences predicted by Theorems~\ref{t_eps}--\ref{t_e=p} are quite fast. Actually, this could be expected on the basis of property~\ref{vw} in Theorem~\ref{eig_props}, where we see that for $|\epsilon|,|\phi|>1$ the pairs $(\epsilon+\epsilon^{-1},\vv_n)$ and $(\phi+\phi^{-1},\ww_n)$ are substantially eigenpairs of $T_{n,\epsilon,\phi}$ already for moderate $n$ due to the exponential convergence to~0 of the error terms $\epsilon^{-n}(\epsilon\phi-1)\ee_n$ and $\phi^{-n}(\epsilon\phi-1)\ee_1$.

\section{Equations for the Eigenvalues and Eigenvectors of $\boldsymbol{T_{n,\epsilon,\phi}}$}\label{eeT}
In this section, we derive equations for the eigenvalues of $T_{n,\epsilon,\phi}$. As we shall see, the equations for the outliers are formally the same as the equations for the non-outliers with the only difference that the trigonometric functions $\sin x$ and $\cos x$ must be replaced by the corresponding hyperbolic functions $\sinh x$ and $\cosh x$. For all $\epsilon,\phi\in\mathbb R$ for which these equations can be solved, one obtains not only the eigenvalues but also the eigenvectors of $T_{n,\epsilon,\phi}$. A special role in the following derivation is played by the theory of linear difference equations \cite{DE}.

Let $\lambda\in\mathbb R$ and $\vv\in\mathbb C^n\backslash\{\mathbf0\}$, so that $(\lambda,\vv)$ is a candidate eigenpair for the real symmetric matrix $T_{n,\epsilon,\phi}$. We have
\begin{align}
T_{n,\epsilon,\phi}\vv=\lambda\vv &\iff \left\{\begin{aligned}
\epsilon v_1+v_2&=\lambda v_1\\
v_{i-1}+v_{i+1}&=\lambda v_i\qquad\forall\,i=2,\ldots,n-1\\
v_{n-1}+\phi v_n&=\lambda v_n
\end{aligned}\right.\notag\\
&\iff\text{exists a sequence }(w_0,w_1,\ldots)\text{ such that }w_i=v_i\text{ for }i=1,\ldots,n\text{ and}\notag\\
&\hphantom{\iff}\ \left\{\begin{aligned}
w_0&=\epsilon w_1\\
w_{i-1}+w_{i+1}&=\lambda w_i\qquad\forall\,i\ge1\label{chain}\\
w_{n+1}&=\phi w_n
\end{aligned}\right.
\end{align}
The characteristic equation of the linear difference equation \eqref{chain} is given by 
\begin{equation}\label{ce}
x^2-\lambda x+1=0.
\end{equation}
We consider five different cases.

\subsection[Case 1: $\lambda\in(-2,2)$]{Case 1: $\boldsymbol{\lambda\in(-2,2)}$}\label{case1}
In this case, we set $\lambda=2\cos\theta$ with $\theta\in(0,\pi)$. 
The roots of the characteristic equation \eqref{ce} are given by
\[ \frac{\lambda\pm\sqrt{\lambda^2-4}}{2}=\frac{2\cos\theta\pm2\i\sin\theta}{2}=\e^{\pm\i\theta}, \]
and they are distinct because $\theta\in(0,\pi)$. The general solution of \eqref{chain} is given by
\[ w_i=A\e^{\i i\theta}+B\e^{-\i i\theta}\qquad\forall\,i\ge0, \]
where $A,B\in\mathbb C$ are arbitrary constants. Keeping in mind that $\vv\ne\mathbf0$, we have
\begin{align*}
T_{n,\epsilon,\phi}\vv=\lambda\vv &\iff\text{exists a sequence }(w_0,w_1,\ldots)\text{ such that }w_i=v_i\text{ for }i=1,\ldots,n\text{ and}\\
&\hphantom{\iff}\ \left\{\begin{aligned}
w_i&=A\e^{\i i\theta}+B\e^{-\i i\theta}\qquad\forall\,i\ge0\\
A+B&=\epsilon A\e^{\i\theta}+\epsilon B\e^{-\i\theta}\\
A\e^{\i(n+1)\theta}+B\e^{-\i(n+1)\theta}&=\phi A\e^{\i n\theta}+\phi B\e^{-\i n\theta}
\end{aligned}\right.\notag\\
&\iff\left\{\begin{aligned}
v_i&=A\e^{\i i\theta}+B\e^{-\i i\theta}\qquad\forall\,i=1,\ldots,n\\
A&=\frac{\epsilon\e^{-\i\theta}-1}{1-\epsilon\e^{\i\theta}}B\qquad\mbox{\footnotesize($1-\epsilon\e^{\i\theta}\ne0$ because $\theta\in(0,\pi)$)}\\
0&=\left|\begin{array}{cc}
1-\epsilon\e^{\i\theta} & 1-\epsilon\e^{-\i\theta}\\
\e^{\i n\theta}(\e^{\i\theta}-\phi) & \e^{-\i n\theta}(\e^{-\i\theta}-\phi)
\end{array}\right|
\end{aligned}\right.\\
&\iff\left\{\begin{aligned}
v_i&=B\Bigl(\frac{\epsilon\e^{-\i\theta}-1}{1-\epsilon\e^{\i\theta}}\e^{\i i\theta}+\e^{-\i i\theta}\Bigr)=\frac{2\i B}{\epsilon\e^{\i\theta}-1}\Bigl(\sin(i\theta)-\epsilon\sin((i-1)\theta)\Bigr)\qquad\forall\,i=1,\ldots,n\\
0&=\sin((n+1)\theta)-(\epsilon+\phi)\sin(n\theta)+\epsilon\phi\sin((n-1)\theta)
\end{aligned}\right.
\end{align*}
We summarize in the next theorem the result that we have obtained.
\begin{theorem}\label{p(-2,2)}
For every $\theta\in(0,\pi)$, the number $\lambda=2\cos\theta$ is an eigenvalue of $T_{n,\epsilon,\phi}$ if and only if
\begin{equation}\label{lambda_expr}
\sin((n+1)\theta)-(\epsilon+\phi)\sin(n\theta)+\epsilon\phi\sin((n-1)\theta)=0.
\end{equation}
In this case, a corresponding eigenvector $\vv=(v_1,\ldots,v_n)$ is given by
\begin{equation}\label{v_expr}
v_i=\sin(i\theta)-\epsilon\sin((i-1)\theta),\qquad i=1,\ldots,n.
\end{equation}
\end{theorem}

\subsection[Case 2: $\lambda\in(2,\infty)$]{Case 2: $\boldsymbol{\lambda\in(2,\infty)}$}\label{case2}
In this case, we set $\lambda=2\cosh\theta$ with $\theta\in(0,\infty)$.
The roots of the characteristic equation \eqref{ce} are given by
\[ \frac{\lambda\pm\sqrt{\lambda^2-4}}{2}=\frac{2\cosh\theta\pm2\sinh\theta}{2}=\e^{\pm\theta}, \]
and they are distinct because $\theta\in(0,\infty)$. The general solution of \eqref{chain} is given by
\[ w_i=A\e^{i\theta}+B\e^{-i\theta}\qquad\forall\,i\ge0, \]
where $A,B\in\mathbb C$ are arbitrary constants. Keeping in mind that $\vv\ne\mathbf0$, we have
\begin{align*}
T_{n,\epsilon,\phi}\vv=\lambda\vv &\iff\text{exists a sequence }(w_0,w_1,\ldots)\text{ such that }w_i=v_i\text{ for }i=1,\ldots,n\text{ and}\\
&\hphantom{\iff}\ \left\{\begin{aligned}
w_i&=A\e^{i\theta}+B\e^{-i\theta}\qquad\forall\,i\ge0\\
A+B&=\epsilon A\e^\theta+\epsilon B\e^{-\theta}\\
A\e^{(n+1)\theta}+B\e^{-(n+1)\theta}&=\phi A\e^{n\theta}+\phi B\e^{-n\theta}
\end{aligned}\right.\notag\\
&\iff\left\{\begin{aligned}
v_i&=A\e^{i\theta}+B\e^{-i\theta}\qquad\forall\,i=1,\ldots,n\\
A+B&=\epsilon A\e^\theta+\epsilon B\e^{-\theta}\qquad\mbox{\footnotesize(this equation is not identically 0 because $1-\epsilon\e^\theta=0\implies1-\epsilon\e^{-\theta}\ne0$)}\\
0&=\left|\begin{array}{cc}
1-\epsilon\e^{\theta} & 1-\epsilon\e^{-\theta}\\
\e^{n\theta}(\e^\theta-\phi) & \e^{-n\theta}(\e^{-\theta}-\phi)
\end{array}\right|
\end{aligned}\right.\\
&\iff\left\{\begin{aligned}
v_i&=A\e^{i\theta}+B\e^{-i\theta}\qquad\forall\,i=1,\ldots,n\\
A+B&=\epsilon A\e^\theta+\epsilon B\e^{-\theta}\\
0&=\sinh((n+1)\theta)-(\epsilon+\phi)\sinh(n\theta)+\epsilon\phi\sinh((n-1)\theta)
\end{aligned}\right.
\end{align*}
\begin{itemize}[nolistsep,leftmargin=*]
	\item If $1-\epsilon\e^\theta=0$, i.e., $\e^{-\theta}=\epsilon$, then the equation $A+B=\epsilon A\e^\theta+\epsilon B\e^{-\theta}$ is equivalent to $B=0$ and so
	\[ T_{n,\epsilon,\phi}\vv=\lambda\vv\iff\left\{\begin{aligned}
v_i&=A\e^{i\theta}=A\epsilon^{-i}\qquad\forall\,i=1,\ldots,n\\
0&=\sinh((n+1)\theta)-(\epsilon+\phi)\sinh(n\theta)+\epsilon\phi\sinh((n-1)\theta)
\end{aligned}\right. \]
	\item If $1-\epsilon\e^\theta\ne0$, then the equation $A+B=\epsilon A\e^\theta+\epsilon B\e^{-\theta}$ is equivalent to $A=\dfrac{\epsilon\e^{-\theta}-1}{1-\epsilon\e^\theta}B$ and so
	\[ T_{n,\epsilon,\phi}\vv=\lambda\vv\iff\left\{\begin{aligned}
v_i&=B\Bigl(\dfrac{\epsilon\e^{-\theta}-1}{1-\epsilon\e^\theta}\e^{i\theta}+\e^{-i\theta}\Bigr)=\frac{2B}{\epsilon\e^\theta-1}\Bigl(\sinh(i\theta)-\epsilon\sinh((i-1)\theta)\Bigr)\qquad\forall\,i=1,\ldots,n\\
0&=\sinh((n+1)\theta)-(\epsilon+\phi)\sinh(n\theta)+\epsilon\phi\sinh((n-1)\theta)
\end{aligned}\right. \]
\end{itemize}
As often happens in mathematics, the ``limit'' case $1-\epsilon\e^\theta=0$ merges with the case $1-\epsilon\e^\theta\ne0$. Indeed, if $1-\epsilon\e^\theta=0$ then $\epsilon=\e^{-\theta}\in(0,1)$ (because $\theta\in(0,\infty)$) and
\begin{align*}
\sinh(i\theta)-\epsilon\sinh((i-1)\theta)=\frac{1-\epsilon^2}{2}\,\epsilon^{-i},\qquad i=1,\ldots,n.
\end{align*}
We summarize in the next theorem the result that we have obtained.

\begin{theorem}\label{p(2,oo)}
For every $\theta\in(0,\infty)$, the number $\lambda=2\cosh\theta$ is an eigenvalue of $T_{n,\epsilon,\phi}$ if and only if
\begin{equation}\label{lambda_exprh}
\sinh((n+1)\theta)-(\epsilon+\phi)\sinh(n\theta)+\epsilon\phi\sinh((n-1)\theta)=0.
\end{equation}
In this case, a corresponding eigenvector $\vv=(v_1,\ldots,v_n)$ is given by
\begin{equation}\label{v_exprh}
v_i=\sinh(i\theta)-\epsilon\sinh((i-1)\theta),\qquad i=1,\ldots,n.
\end{equation}
\end{theorem}

\subsection[Case 3: $\lambda=2$]{Case 3: $\boldsymbol{\lambda=2}$}\label{case3}
In this case, the characteristic equation \eqref{ce} has only one root $x=1$ with multiplicity 2. The general solution of \eqref{chain} is given by
\[ w_i=A+Bi\qquad\forall\,i\ge0, \]
where $A,B$ are arbitrary constants. Keeping in mind that $\vv\ne\mathbf0$, we have
\begin{align*}
T_{n,\epsilon,\phi}\vv=\lambda\vv &\iff\text{exists a sequence }(w_0,w_1,\ldots)\text{ such that }w_i=v_i\text{ for }i=1,\ldots,n\text{ and}\\
&\hphantom{\iff}\ \left\{\begin{aligned}
w_i&=A+Bi\qquad\forall\,i\ge0\\
A&=\epsilon A+\epsilon B\\
A+B(n+1)&=\phi A+\phi Bn
\end{aligned}\right.\notag\\
&\iff\left\{\begin{aligned}
v_i&=A+Bi\qquad\forall\,i=1,\ldots,n\\
A&=\epsilon A+\epsilon B\qquad\mbox{\footnotesize(this equation is not identically 0 because $1-\epsilon=0\implies\epsilon\ne0$)}\\
0&=\left|\begin{array}{cc}
1-\epsilon & -\epsilon\\
1-\phi & n+1-\phi n
\end{array}\right|
\end{aligned}\right.\\
&\iff\left\{\begin{aligned}
v_i&=A+Bi\qquad\forall\,i=1,\ldots,n\\
A&=\epsilon A+\epsilon B\\
0&=n+1-(\epsilon+\phi)n+\epsilon\phi(n-1)
\end{aligned}\right.
\end{align*}
\begin{itemize}[nolistsep,leftmargin=*]
	\item If $1-\epsilon=0$, then the equation $A=\epsilon A+\epsilon B$ is equivalent to $B=0$ and so
	\[ T_{n,\epsilon,\phi}\vv=\lambda\vv\iff\left\{\begin{aligned}
v_i&=A\qquad\forall\,i=1,\ldots,n\\
0&=n+1-(\epsilon+\phi)n+\epsilon\phi(n-1)
\end{aligned}\right. \]
	\item If $1-\epsilon\ne0$, then the equation $A=\epsilon A+\epsilon B$ is equivalent to $A=\dfrac{\epsilon}{1-\epsilon}B$ and so
	\[ T_{n,\epsilon,\phi}\vv=\lambda\vv\iff\left\{\begin{aligned}
v_i&=B\Bigl(\dfrac{\epsilon}{1-\epsilon}+i\Bigr)=\frac{B}{1-\epsilon}\Bigl(\epsilon+(1-\epsilon)i\Bigr)\qquad\forall\,i=1,\ldots,n\\
0&=n+1-(\epsilon+\phi)n+\epsilon\phi(n-1)
\end{aligned}\right. \]
\end{itemize}
The case $1-\epsilon=0$ merges with the case $1-\epsilon\ne0$, because if $1-\epsilon=0$ then $\epsilon=1$ and
\begin{align*}
\epsilon+(1-\epsilon)i=\epsilon,\qquad i=1,\ldots,n.
\end{align*}
We summarize in the next theorem the result that we have obtained.

\begin{theorem}\label{p=2}
The number $\lambda=2$ is an eigenvalue of $T_{n,\epsilon,\phi}$ if and only if
\begin{equation}\label{lambda_expr2}
n+1-(\epsilon+\phi)n+\epsilon\phi(n-1)=0.
\end{equation}
In this case, a corresponding eigenvector $\vv=(v_1,\ldots,v_n)$ is given by
\begin{equation}\label{v_expr2}
v_i=\epsilon+(1-\epsilon)i,\qquad i=1,\ldots,n.
\end{equation}
\end{theorem}

\begin{remark}\label{(0,pi)->2}
Note that
\begin{align*}
&\lim_{\theta\to0}\frac{\sin((n+1)\theta)-(\epsilon+\phi)\sin(n\theta)+\epsilon\phi\sin((n-1)\theta)}{\sin\theta}\\[8pt]
&\quad=\lim_{\theta\to0}\frac{\sinh((n+1)\theta)-(\epsilon+\phi)\sinh(n\theta)+\epsilon\phi\sinh((n-1)\theta)}{\sinh\theta}=n+1-(\epsilon+\phi)n+\epsilon\phi(n-1),\\[8pt]
&\lim_{\theta\to0}\frac{\sin(i\theta)-\epsilon\sin((i-1)\theta)}{\sin\theta}=\lim_{\theta\to0}\frac{\sinh(i\theta)-\epsilon\sinh((i-1)\theta)}{\sinh\theta}=\epsilon+(1-\epsilon)i.
\end{align*}
This shows that \eqref{lambda_expr2}--\eqref{v_expr2} can be obtained from \eqref{lambda_expr}--\eqref{v_expr} (resp., \eqref{lambda_exprh}--\eqref{v_exprh}) by division by $\sin\theta$ (resp., $\sinh\theta$). In particular, we could reformulate Theorems~\ref{p(-2,2)}--\ref{p(2,oo)} to include Theorem~\ref{p=2}.
\end{remark}

\subsection[Case 4: $\lambda\in(-\infty,-2)$]{Case 4: $\boldsymbol{\lambda\in(-\infty,-2)}$}\label{case4}
In this case, we set $\lambda=-2\cosh\theta$ with $\theta\in(0,\infty)$. The derivation is essentially the same as in Section~\ref{case2}; we leave the details to the reader and we report the analog of Theorem~\ref{p(2,oo)}.

\begin{theorem}\label{p(-oo,-2)}
For every $\theta\in(0,\infty)$, the number $\lambda=-2\cosh\theta$ is an eigenvalue of $T_{n,\epsilon,\phi}$ if and only if
\begin{equation}\label{lambda_exprh'}
\sinh((n+1)\theta)+(\epsilon+\phi)\sinh(n\theta)+\epsilon\phi\sinh((n-1)\theta)=0.
\end{equation}
In this case, a corresponding eigenvector $\vv=(v_1,\ldots,v_n)$ is given by
\begin{equation}\label{v_exprh'}
v_i=(-1)^i\bigl(\sinh(i\theta)+\epsilon\sinh((i-1)\theta)\bigr),\qquad i=1,\ldots,n.
\end{equation}
\end{theorem}

\subsection[Case 5: $\lambda=-2$]{Case 5: $\boldsymbol{\lambda=-2}$}\label{case5}
The derivation is essentially the same as in Section~\ref{case3}; we leave the details to the reader and we report the analog of Theorem~\ref{p=2}.

\begin{theorem}\label{p=-2}
The number $\lambda=-2$ is an eigenvalue of $T_{n,\epsilon,\phi}$ if and only if
\begin{equation}\label{lambda_expr2'}
n+1+(\epsilon+\phi)n+\epsilon\phi(n-1)=0.
\end{equation}
In this case, a corresponding eigenvector $\vv=(v_1,\ldots,v_n)$ is given by
\begin{equation}\label{v_expr2'}
v_i=(-1)^i\bigl(-\epsilon+(1+\epsilon)i\bigr),\qquad i=1,\ldots,n.
\end{equation}
\end{theorem}

\begin{remark}\label{(0,pi)->-2}
Note that
\begin{align*}
&\lim_{\theta\to\pi}(-1)^n\frac{\sin((n+1)\theta)-(\epsilon+\phi)\sin(n\theta)+\epsilon\phi\sin((n-1)\theta)}{\sin\theta}\\[8pt]
&\quad=\lim_{\theta\to0}\frac{\sinh((n+1)\theta)+(\epsilon+\phi)\sinh(n\theta)+\epsilon\phi\sinh((n-1)\theta)}{\sinh\theta}=n+1+(\epsilon+\phi)n+\epsilon\phi(n-1),\\[8pt]
&-\lim_{\theta\to\pi}\frac{\sin(i\theta)-\epsilon\sin((i-1)\theta)}{\sin\theta}=\lim_{\theta\to0}\frac{(-1)^i\bigl(\sinh(i\theta)+\epsilon\sinh((i-1)\theta)\bigr)}{\sinh\theta}=(-1)^i\bigl(-\epsilon+(1+\epsilon)i\bigr).
\end{align*}
This shows that \eqref{lambda_expr2'}--\eqref{v_expr2'} can be obtained from \eqref{lambda_expr}--\eqref{v_expr} (resp., \eqref{lambda_exprh'}--\eqref{v_exprh'}) by division by $\sin\theta$ (resp., $\sinh\theta$). In particular, we could reformulate Theorems~\ref{p(-2,2)} and~\ref{p(-oo,-2)} to include Theorem~\ref{p=-2}.
\end{remark}

\section{Eigendecomposition of $\boldsymbol{T_{n,\epsilon,\phi}}$ for Specific Choices of $\boldsymbol{\epsilon}$ and $\boldsymbol{\varphi}$}\label{sec:sep}

\subsection[$\epsilon,\phi\in\{0,1,-1\}$]{$\boldsymbol{\epsilon,\phi\in\{0,1,-1\}}$}\label{01-1}
As noted in the introduction, the eigendecomposition of $T_{n,\epsilon,\phi}$ for $\epsilon,\phi\in\{0,1,-1\}$ is already available in the literature.
The purpose of this section is simply to show that it can also be obtained from Theorems~\ref{p(-2,2)}, \ref{p=2}, and~\ref{p=-2}. Note that Theorems~\ref{p(2,oo)} and~\ref{p(-oo,-2)} are useless in this case as $T_{n,\epsilon,\phi}$ does not have outliers for $\epsilon,\phi\in\{0,1,-1\}$; see Theorem~\ref{eig_props}. 

For $(\epsilon,\phi)=(0,0)$, Theorem~\ref{p(-2,2)} immediately yields the eigenpairs $(\lambda_k,\vv^{(k)})$, $k=1,\ldots,n$, with
\[ \lambda_k=2\cos\theta_k,\qquad\vv^{(k)}=\bigl[\sin(i\theta_k)\bigr]_{i=1}^n,\qquad\theta_k=\frac{k\pi}{n+1}. \]
For $(\epsilon,\phi)=(1,1)$, using sine addition/subtraction formulas, we see that equation \eqref{lambda_expr} is equivalent to 
\[ \sin(n\theta)(2\cos\theta-2)=0, \]
whose solutions in $(0,\pi)$ are $\theta_k=\frac{k\pi}{n}$, $k=1,\ldots,n-1$; moreover, equation \eqref{lambda_expr2} is satisfied. Since, by prosthaphaeresis formulas, 
\[ \sin(i\theta)-\sin((i-1)\theta)=2\sin\frac\theta2\cos\frac{(2i-1)\theta}{2}, \]
we conclude by Theorems~\ref{p(-2,2)} and~\ref{p=2} that, for $(\epsilon,\phi)=(1,1)$, a complete set of eigenpairs for $T_{n,\epsilon,\phi}=T_{n,1,1}$ is given by $(\lambda_k,\vv^{(k)})$, $k=0,\ldots,n-1$, with
\[ \lambda_k=2\cos\theta_k,\qquad\vv^{(k)}=\biggl[\cos\frac{(2i-1)\theta_k}{2}\biggr]_{i=1}^n,\qquad\theta_k=\frac{k\pi}n. \]
Similar derivations, using sine addition/subtraction and prosthaphaeresis formulas, can be done for all $\epsilon,\phi\in\{0,1,-1\}$; we leave the details to the reader.

\subsection[$\epsilon\phi=1$]{$\boldsymbol{\epsilon\phi=1}$}\label{prod1}
We focus in this section on the case $\epsilon\phi=1$, which is crucial for the applications presented in Section~\ref{sec:app}. To the best of the authors' knowledge, this case has never been addressed in the literature. 
Besides $\epsilon\phi=1$, we also assume that:
\begin{itemize}[nolistsep,leftmargin=*]
	\item $\epsilon,\phi>0$ (because no additional difficulties are encountered if $\epsilon,\phi<0$);
	\item $\epsilon,\phi\ne1$ (because the case $\epsilon=\phi=1$ has already been addressed in Section~\ref{01-1}).
\end{itemize}
Under these assumptions, we have
\[ \frac{\epsilon+\phi}{2}=\frac{\epsilon+\epsilon^{-1}}{2}=\cosh(\log\epsilon)>1. \]
Using sine addition/subtraction formulas, we see that equation \eqref{lambda_expr} is equivalent to
\[ \sin(n\theta)\Bigl(\cos\theta-\frac{\epsilon+\epsilon^{-1}}2\Bigr)=0, \]
whose solutions in $(0,\pi)$ are $\theta_k=\frac{k\pi}n$, $k=1,\ldots,n-1$. Thus, Theorem~\ref{p(-2,2)} yields $n-1$ eigenpairs of $T_{n,\epsilon,\phi}$, i.e., $(\lambda_k,\vv^{(k)})$, $k=1,\ldots,n-1$, with
\[ \lambda_k=2\cos\theta_k,\qquad\vv^{(k)}=\bigl[\sin(i\theta_k)-\epsilon\sin((i-1)\theta_k)\bigr]_{i=1}^n,\qquad\theta_k=\frac{k\pi}n. \]
We still have to find one eigenvalue, which can be neither $2$ nor $-2$ because, under our assumptions, equations \eqref{lambda_expr2} and \eqref{lambda_expr2'} are not satisfied. In other words, the eigenvalue we are looking for is an outlier. Since equation \eqref{lambda_exprh} is equivalent to
\[ \sinh(n\theta)\Bigl(\cosh\theta-\frac{\epsilon+\epsilon^{-1}}2\Bigr)=0, \]
it has a unique solution in $(0,\infty)$ given by $\theta=|\hspace{-1pt}\log\epsilon|$. We then obtain the outlier $\lambda$ and the corresponding eigenvector $\vv$ from Theorem~\ref{p(2,oo)}:
\[ \lambda=2\cosh\theta,\qquad\vv=\bigl[\sinh(i\theta)-\epsilon\sinh((i-1)\theta)\bigr]_{i=1}^n,\qquad\theta=|\hspace{-1pt}\log\epsilon|. \]
After straightforward manipulations, involving also a renormalization of $\vv$, we get for the outlier eigenpair $(\lambda,\vv)$ the following simplified expressions:
\[ \lambda=\epsilon+\epsilon^{-1}=\phi+\phi^{-1}=\epsilon+\phi,\qquad\vv=\bigl[\epsilon^{-i+1}\bigr]_{i=1}^n=\bigl[\phi^{i-1}\bigr]_{i=1}^n. \]
Note that this outlier eigenpair could also be obtained from property~\ref{vw} of Theorem~\ref{eig_props}.
In conclusion, if we set
\[ V_{n,\epsilon}=\bigl[\,\vv\ \big|\ \vv^{(1)}\ \big|\ \cdots\ \vv^{(n-1)}\,\bigr], \]
then the eigendecomposition of $T_{n,\epsilon,\phi}$ is given by
\[ T_{n,\epsilon,\phi}=V_{n,\epsilon}\begin{bmatrix}
\epsilon+\epsilon^{-1} & & & & \\
& 2\cos\frac\pi n & & & \\
& & 2\cos\frac{2\pi}n & & \\
& & & \ddots & \\
& & & & 2\cos\frac{(n-1)\pi}n 
\end{bmatrix}V_{n,\epsilon}^{-1}. \]

\section{Applications}\label{sec:app}

In this section, we present a few applications of our results in the context of Markov chains and processes. Section~\ref{queue} deals with a queuing model. Sections~\ref{RW1} and~\ref{RWd} are devoted to random walks in unidimensional and multidimensional lattices, respectively. Finally, Sections~\ref{mdp} and~\ref{sec:economics} focus on multidimensional reflected diffusion processes and related economics applications.

\subsection{Queuing Model}\label{queue}

Consider a continuous-time Markov chain with $n$ states $0,\ldots,n-1$ and with transition rate matrix (infinitesimal generator) given by
\begin{equation}\label{Qnlm}
Q_{n,\lambda,\mu}=\begin{bmatrix}
-\lambda & \lambda & & & \\[5pt]
\mu & -\lambda-\mu & \lambda & & \\[5pt]
& \ddots & \ddots & \ddots & \\[5pt]
& & \mu & -\lambda-\mu & \lambda \\[5pt]
& & & \mu & -\mu
\end{bmatrix},
\end{equation}
where $\lambda,\mu>0$. Markov chains of this kind are referred to as M/M/1/$K$ queues (with $K=n-1$). They find applications in queuing theory \cite{BiniCIME,Giambene,lawler2006introduction}, especially in telecommunications \cite[Section~5.7]{Giambene}.
In this section, we derive the eigendecomposition of 
\begin{equation*}
Q_{n,\lambda,\mu}^\top=\begin{bmatrix}
-\lambda & \mu & & & \\[5pt]
\lambda & -\lambda-\mu & \mu & & \\[5pt]
& \ddots & \ddots & \ddots & \\[5pt]
& & \lambda & -\lambda-\mu & \mu \\[5pt]
& & & \lambda & -\mu
\end{bmatrix}.
\end{equation*}
We begin with the following lemma, which can be proved by direct computation.

\begin{lemma}\label{ts}
Let
\[ T=\begin{bmatrix}
a_1 & b_1 & & & \\
c_1 & a_2 & b_2 & & \\
& c_2 & \ddots & \ddots & \\
& & \ddots & \ddots & b_{n-1}\\
& & & c_{n-1} & a_n
\end{bmatrix} \]
be a real tridiagonal matrix such that $b_ic_i>0$ for all $i=1,\ldots,n-1$. Then
\[ T=D\begin{bmatrix}
a_1 & \sqrt{b_1c_1} & & & \\[8pt]
\sqrt{b_1c_1} & a_2 & \sqrt{b_2c_2} & & \\[8pt]
& \sqrt{b_2c_2} & \ddots & \ddots & \\[8pt]
& & \ddots & \ddots & \sqrt{b_{n-1}c_{n-1}}\\[8pt]
& & & \sqrt{b_{n-1}c_{n-1}} & a_n
\end{bmatrix}D^{-1}, \]
where $\displaystyle D={\rm diag}\biggl(1,\sqrt{\frac{c_1}{b_1}},\sqrt{\frac{c_1c_2}{b_1b_2}},\ldots,\sqrt{\frac{c_1\cdots c_{n-1}}{b_1\cdots b_{n-1}}}\,\biggr)$. 
\end{lemma}

By applying Lemma~\ref{ts} to the matrix $Q_{n,\lambda,\mu}^\top$, we obtain
\begin{equation*}
Q_{n,\lambda,\mu}^\top=D_{n,\lambda,\mu}X_{n,\lambda,\mu}D_{n,\lambda,\mu}^{-1},
\end{equation*}
where 
\begin{align*}
D_{n,\lambda,\mu}&={\rm diag}(1,\tau,\tau^2,\ldots,\tau^{n-1}),\qquad\tau=\sqrt{\frac\lambda\mu},\\
X_{n,\lambda,\mu}&=\begin{bmatrix}
-\lambda & \sqrt{\lambda\mu} & & & \\[5pt]
\sqrt{\lambda\mu} & -\lambda-\mu & \sqrt{\lambda\mu} & & \\[5pt]
& \ddots & \ddots & \ddots & \\[5pt]
& & \sqrt{\lambda\mu} & -\lambda-\mu & \sqrt{\lambda\mu}\\[5pt]
& & & \sqrt{\lambda\mu} & -\mu
\end{bmatrix}.
\end{align*}
A direct verification shows that
\begin{equation*}
X_{n,\lambda,\mu}=(-\lambda-\mu)I_n+\sqrt{\lambda\mu}\,T_{n,\epsilon,\phi}\qquad{\rm with}\qquad\begin{cases}\epsilon=\tau^{-1}=\sqrt{\dfrac\mu\lambda},\\[10pt] \phi=\tau=\sqrt{\dfrac\lambda\mu}.\end{cases}
\end{equation*}
Since $\epsilon\phi=1$, the eigendecomposition of $X_{n,\lambda,\mu}$ (and hence also of $Q_{n,\lambda,\mu}^\top$) is immediately obtained from the results in Section~\ref{prod1}.
In particular, the eigenpairs of $Q_{n,\lambda,\mu}^\top$ are given by $(\nu_k,\ww_k)$, $k=0,\ldots,n-1$, where
\begin{alignat}{5}
\nu_0&=0, &\qquad\ww_0&=\bigl[1,\tau^2,\tau^4,\ldots,\tau^{2n-2}\bigr]^\top,\label{nu0w0}
\end{alignat}
and, for $k=1,\ldots,n-1$,
\begin{alignat}{5}
\nu_k&=-\lambda-\mu+2\sqrt{\lambda\mu}\cos\theta_k, &\qquad\ww_k&=\bigl[\tau^{i-1}\sin(i\theta_k)-\tau^{i-2}\sin((i-1)\theta_k)\bigr]_{i=1}^n, &\qquad\theta_k&=\frac{k\pi}n.\label{nukwk}
\end{alignat}

\begin{remark}[\textbf{Steady-State Distribution}]
Since
\[ \|\ww_0\|_1=\sum_{i=0}^{n-1}\tau^{2i}=\frac{1-\tau^{2n}}{1-\tau^2}=\frac{1-\rho^n}{1-\rho},\qquad\rho=\tau^2, \]
the steady-state (or stationary/limiting) distribution of the considered queuing model, i.e., the normalized positive eigenvector of $Q_{n,\lambda,\mu}^\top$ associated with the eigenvalue~0, is given by
\[ \frac{\ww_0}{\|\ww_0\|_1} = \frac{1-\rho}{1-\rho^n}\bigl[1,\rho,\rho^2,\ldots,\rho^{n-1}\bigr]^\top, \]
where it is understood that in the case $\rho=1$ we take the limit $\rho\to1$. For a different derivation of this result, see \cite[Section~5.7]{Giambene}.
\end{remark}

\begin{remark}[\textbf{Second Eigenvalue}]\label{2nd_eig}
It is clear from \eqref{nukwk} and the geometric-arithmetic mean inequality $\sqrt{\lambda\mu}\le\frac12(\lambda+\mu)$ that all nonzero eigenvalues of $Q_{n,\lambda,\mu}^\top$ are negative. The largest of them, i.e., the second largest eigenvalue after 0, is $\nu_1=-\lambda-\mu+2\sqrt{\lambda\mu}\cos\frac\pi n$. The second eigenvalue gives information about the convergence speed towards the steady-state distribution of power methods \cite[p.~371]{Bini}; see also \cite{gabaix2009power} and \cite[Section~7.2]{lawler2006introduction}. We will return to the role of the second eigenvalue in Section~\ref{sec:economics}.
\end{remark}

\begin{remark}
The above derivation of the eigendecomposition of $Q_{n,\lambda,\mu}^\top$ requires only the hypothesis $\lambda\mu>0$. In other words, the eigendecomposition of $Q_{n,\lambda,\mu}^\top$ is given by \eqref{nu0w0}--\eqref{nukwk} for all $\lambda,\mu\in\mathbb R$ such that $\lambda\mu>0$.
\end{remark}

\subsection{Random Walk in a Unidimensional Lattice}\label{RW1}

\begin{figure}
\centering
\includegraphics[width=0.40\textwidth]{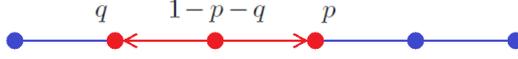}
\caption{Random walk in a unidimensional lattice.}
\label{RandomWalkGraph1D}
\end{figure}

Consider a discrete-time Markov chain with $n$ states $1,\ldots,n$ and with matrix of transition probabilities given by
\begin{equation}\label{Pnpq}
P_{n,p,q}=\begin{bmatrix}
1-p & p & & & \\[5pt]
q & 1-p-q & p & & \\[5pt]
& \ddots & \ddots & \ddots & \\[5pt]
& & q & 1-p-q & p \\[5pt]
& & & q & 1-q
\end{bmatrix},
\end{equation}
where $p,q>0$ and $p+q\le1$.
Markov chains of this kind are often referred to as random walks in the unidimensional lattice $\{1,\ldots,n\}$; see Figure~\ref{RandomWalkGraph1D}. 
The difference with respect to traditional random walks in $\mathbb Z$ is that states~1 and~$n$ act as absorbing/reflecting barriers: when the system is in state~1, it cannot go to a hypothetical previous state 0 with probability $q$ (as it happens for all other states $2,\ldots,n$), because the probability $q$ of going to a previous state~0 is absorbed in the probability of staying in state~1, which grows from $1-p-q$ to $1-p$; a similar discussion applies to state~$n$.

In this section, we derive the eigendecomposition of
\[ P_{n,p,q}^\top=\begin{bmatrix}
1-p & q & & & \\[5pt]
p & 1-p-q & q & & \\[5pt]
& \ddots & \ddots & \ddots & \\[5pt]
& & p & 1-p-q & q \\[5pt]
& & & p & 1-q
\end{bmatrix}. \]
To this end, simply note that
\[ P_{n,p,q}^\top = I_n+Q_{n,p,q}^\top, \]
where $Q_{n,p,q}$ is given by \eqref{Qnlm} for $(\lambda,\mu)=(p,q)$. By \eqref{nu0w0}--\eqref{nukwk}, the eigenpairs of $P_{n,p,q}^\top$ are given by $(\mu_k,\ww_k)$, $k=0,\ldots,n-1$, where
\begin{alignat}{5}
\mu_0&=1, &\qquad\ww_0&=\bigl[1,\alpha^2,\alpha^4,\ldots,\alpha^{2n-2}\bigr]^\top,\label{mu0w0}
\end{alignat}
and, for $k=1,\ldots,n-1$,
\begin{alignat}{5}
\mu_k&=1-p-q+2\sqrt{pq}\cos\theta_k, &\qquad\ww_k&=\bigl[\alpha^{i-1}\sin(i\theta_k)-\alpha^{i-2}\sin((i-1)\theta_k)\bigr]_{i=1}^n, &\qquad\theta_k&=\frac{k\pi}n,\label{mukwk}
\end{alignat}
with
\[ \alpha=\sqrt{\frac pq}. \]

\begin{remark}[\textbf{Steady-State Distribution}]
Since
\[ \|\ww_0\|_1=\sum_{i=0}^{n-1}\alpha^{2i}=\frac{1-\alpha^{2n}}{1-\alpha^2}=\frac{1-\beta^n}{1-\beta},\qquad\beta=\alpha^2, \]
the steady-state distribution of the unidimensional random walk, i.e., the normalized positive eigenvector of $P_{n,p,q}^\top$ associated with the eigenvalue~1, is given by
\[ \frac{\ww_0}{\|\ww_0\|_1} = \frac{1-\beta}{1-\beta^n}\bigl[1,\beta,\beta^2,\ldots,\beta^{n-1}\bigr]^\top, \]
where it is understood that in the case $\beta=1$ we take the limit $\beta\to1$.
\end{remark}

\subsection{Random Walk in a Multidimensional Lattice}\label{RWd}

\begin{figure}
\centering
\includegraphics[width=0.60\textwidth]{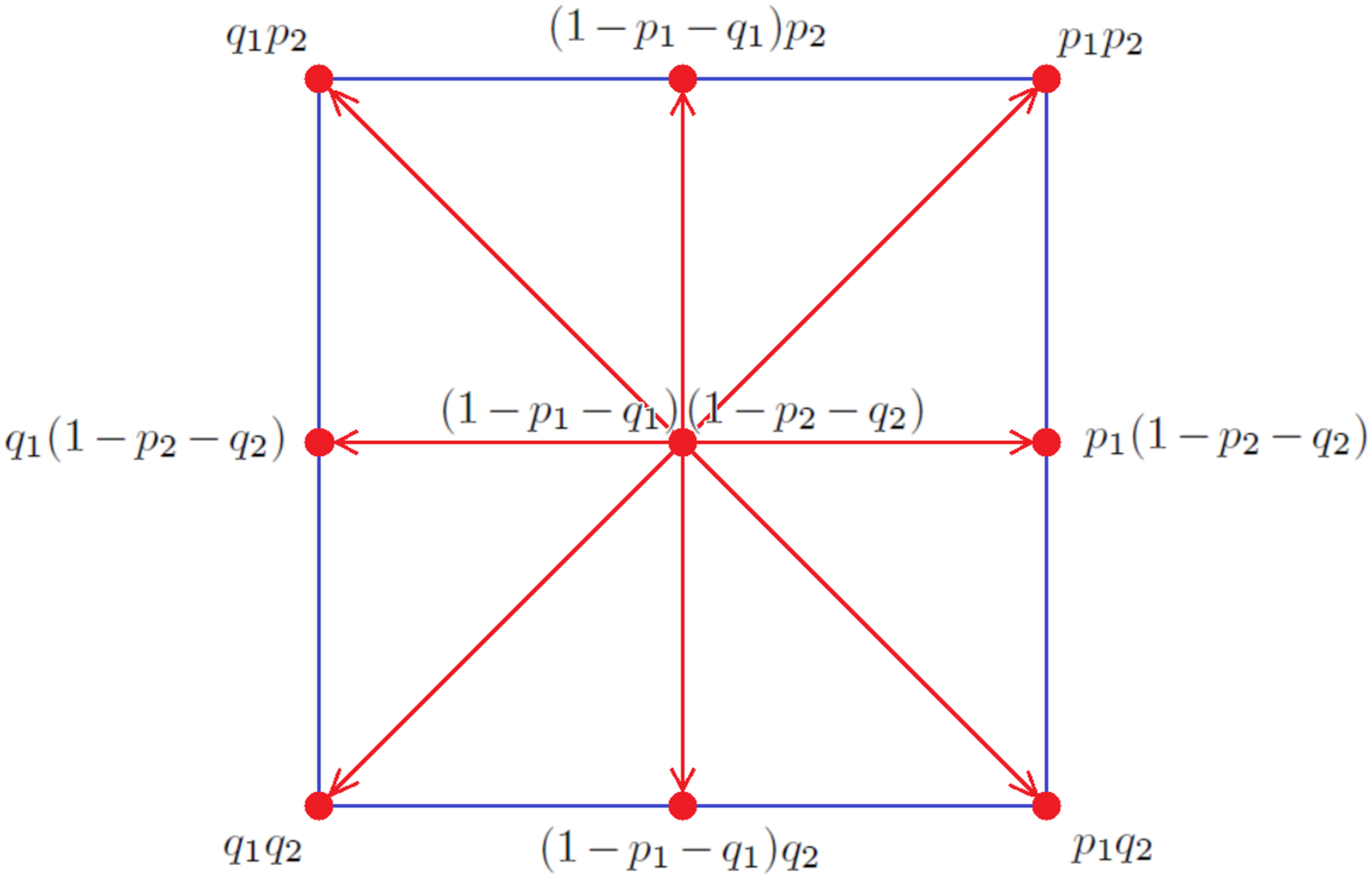}
\caption{Random walk in a bidimensional lattice.}
\label{RandomWalkGraphMultiD}
\end{figure}

Let $\nn=(n_1,\ldots,n_d)\in\mathbb N^d$ and $N(\nn)=\prod_{i=1}^dn_i$. We denote by $\{\mathbf1,\ldots,\nn\}$ the multi-index range $\{\ii\in\mathbb N^d:\:\mathbf1\le\ii\le\nn\}$, where $\mathbf1=(1,\ldots,1)$ and inequalities between vectors such as $\ii\le\nn$ must be interpreted componentwise. When writing $\ii=\mathbf1,\ldots,\nn$, we mean that $\ii$ varies from $\mathbf1$ to $\nn$ over the multi-index range $\{\mathbf1,\ldots,\nn\}$ following the standard lexicographic ordering:
\[ \Bigl[\ \ldots\ \bigl[\ [\:(i_1,\ldots,i_d)\:]_{i_d=1,\ldots,n_d}\ \bigr]_{i_{d-1}=1,\ldots,n_{d-1}}\ \ldots\ \Bigr]_{i_1=1,\ldots,n_1}. \]
We refer the reader to \cite[Section~2.1.2]{GLTbookII} for more details on the multi-index notation.

Consider a discrete-time Markov chain with $N(\nn)$ states $\mathbf1,\ldots,\nn$ and with matrix of transition probabilities
\[ P_{\nn,\pp,\qq}=\bigotimes_{r=1}^dP_{n_r,p_r,q_r}, \]
where 
\begin{itemize}[leftmargin=*,nolistsep]
	\item $\pp=(p_1,\ldots,p_d)$ and $\qq=(q_1,\ldots,q_d)$ satisfy $\pp,\qq>\mathbf0$ and $\pp+\qq\le\mathbf1$, 
	\item the matrix $P_{n_r,p_r,q_r}$ is defined by \eqref{Pnpq} for $(n,p,q)=(n_r,p_r,q_r)$,
	\item $\bigotimes$ denotes the tensor (Kronecker) product.
\end{itemize}
Markov chains of this kind are often referred to as random walks in the $d$-dimensional lattice $\{\mathbf1,\ldots,\nn\}$.
They are a generalization of the unidimensional random walks discussed in Section~\ref{RW1}.
By the properties of tensor products \cite[Section~2.5]{GLTbookII}, for all $\ii,\jj=\mathbf1,\ldots,\nn$, the probability of going from state~$\ii$ to state~$\jj$ is given by
\[ (P_{\nn,\pp,\qq})_{\ii\jj}=\prod_{r=1}^d(P_{n_r,p_r,q_r})_{i_rj_r}, \]
and it is equal to the product for $r=1,\ldots,d$ of the probability of going from state~$i_r$ to state~$j_r$ in a unidimensional random walk with transition matrix $P_{n_r,p_r,q_r}$ as considered in Section~\ref{RW1}. In short, a $d$-dimensional random walk is the result of $d$ independent unidimensional random walks (one for each space dimension); see Figure~\ref{RandomWalkGraphMultiD} for a bidimensional illustration.

By the properties of tensor products and the results of Section~\ref{RW1}, we can immediately obtain the eigendecomposition of $P_{\nn,\pp,\qq}^\top$. In particular, the eigenpairs of $P_{\nn,\pp,\qq}^\top$ are given by $(\mu_\kk,\ww_\kk)$, $\kk=\mathbf0,\ldots,\nn-\mathbf1$, where
\[ \mu_\kk=\prod_{r=1}^d\mu_{k_r},\qquad\ww_\kk=\bigotimes_{r=1}^d\ww_{k_r}, \]
and $(\mu_{k_r},\ww_{k_r})$ is defined by \eqref{mu0w0}--\eqref{mukwk} for $(k,n,p,q,\alpha)=(k_r,n_r,p_r,q_r,\alpha_r)$ with $\alpha_r=\sqrt{p_r/q_r}$. 

\begin{remark}[\textbf{Steady-State Distribution}]
The steady-state distribution of the $d$-dimensional random walk is given by
\[ \frac{\ww_{\mathbf0}}{\|\ww_{\mathbf0}\|_1}=\bigotimes_{r=1}^d\frac{1-\beta_r}{1-\beta_r^{n_r}}\bigl[1,\beta_r,\beta_r^2,\ldots,\beta_r^{n_r-1}\bigr]^\top,\qquad\beta_r=\alpha_r^2, \]
i.e., it is the tensor product of the steady-state distributions of the individual unidimensional random walks that compose it.
\end{remark}

\subsection{Multidimensional Diffusion Processes}\label{mdp}
Consider a $d$-dimensional diffusion process, where the diffusions in each dimension are independent of each other and subject to a reflecting boundary condition at each side.
We assume for simplicity that, for every $r=1,\ldots,d$, the direction $x_r$ is discretized uniformly with $n_r$ nodes separated by a discretization step $\Delta_r>0$.
This discretization gives rise to a $n_1\times\cdots\times n_d$ lattice whose points $\xx_\ii$ are naturally indexed by a multi-index $\ii=\mathbf1,\ldots,\nn$, with $\nn=(n_1,\ldots,n_d)$.
The diffusion in direction $x_r$ is a Brownian motion characterized by two parameters: a drift $\mu_r\in\mathbb R$ and a variance $\sigma_r^2>0$.
For the direction $x_r$, the infinitesimal generator $L_{n_r,\mu_r,\sigma_r}$ coincides with the generator of a 1-dimensional diffusion process with drift $\mu_r$ and variance $\sigma_r^2$ discretized uniformly with $n_r$ nodes separated by a discretization step $\Delta_r$. In formulas, $L_{n_r,\mu_r,\sigma_r}$ is an $n_r\times n_r$ matrix, which in the case $\mu_r\le0$ is given by
\begin{align*}
L_{n_r,\mu_r,\sigma_r} &= \frac{\mu_r}{\Delta_r}\left[\begin{array}{rrrrr}
0 & & & & \\
-1 & 1 & & & \\
 & -1 & 1 &  & \\
& & \ddots & \ddots & \\
& & & -1 & 1
\end{array}\right] +
\frac{\sigma_r^2}{2\Delta_r^2}\left[\begin{array}{rrrrr}
-1 & 1 & & & \\
1 & -2 & 1 & & \\
& \ddots & \ddots & \ddots & \\
& & 1 & -2 & 1 \\
& & & 1 & -1 \\
\end{array}\right]=Q_{n_r,\tilde\lambda_r,\tilde\mu_r},
\end{align*}
where $\displaystyle\tilde\lambda_r=\frac{\sigma_r^2}{2\Delta_r^2}$, $\displaystyle\tilde\mu_r=\frac{\sigma_r^2}{2\Delta_r^2}-\frac{\mu_r}{\Delta_r}$, and $Q_{n,\lambda,\mu}$ is defined in \eqref{Qnlm}; and in the case $\mu_r\ge0$ is given by
\begin{align*}
L_{n_r,\mu_r,\sigma_r} &= \frac{\mu_r}{\Delta_r}\left[\begin{array}{rrrrr}
-1 & 1& & & \\
 & -1 & 1& & \\
 & & \ddots & \ddots  & \\
& &  & -1 & 1 \\
& & &  & 0
\end{array}\right] +
\frac{\sigma_r^2}{2\Delta_r^2}\left[\begin{array}{rrrrr}
-1 & 1 & & & \\
1 & -2 & 1 & & \\
& \ddots & \ddots & \ddots & \\
& & 1 & -2 & 1 \\
& & & 1 & -1 \\
\end{array}\right]=Q_{n_r,\tilde\lambda_r,\tilde\mu_r},
\end{align*}
where $\displaystyle\tilde\lambda_r=\frac{\sigma_r^2}{2\Delta_r^2}+\frac{\mu_r}{\Delta_r}$, $\displaystyle\tilde\mu_r=\frac{\sigma_r^2}{2\Delta_r^2}$, and $Q_{n,\lambda,\mu}$ is defined in \eqref{Qnlm}. In short,
\[ L_{n_r,\mu_r,\sigma_r}=Q_{n_r,\tilde\lambda_r,\tilde\mu_r},\qquad\tilde\lambda_r=\left\{\begin{aligned}&\frac{\sigma_r^2}{2\Delta_r^2}, &&&\mbox{if $\mu_r\le0$},\\ &\frac{\sigma_r^2}{2\Delta_r^2}+\frac{\mu_r}{\Delta_r}, &&&\mbox{if $\mu_r\ge0$},\end{aligned}\right.\qquad\tilde\mu_r=\left\{\begin{aligned}&\frac{\sigma_r^2}{2\Delta_r^2}-\frac{\mu_r}{\Delta_r}, &&&\mbox{if $\mu_r\le0$},\\ &\frac{\sigma_r^2}{2\Delta_r^2}, &&&\mbox{if $\mu_r\ge0$}.\end{aligned}\right. \]
The differential operator (infinitesimal generator) of the $d$-dimensional diffusion process is given by
\begin{align}\label{eq:diffusion-generator}
L_{\nn,\bmu,\bsigma} &= \sum_{i=1}^d I_{n_1}\otimes\cdots\otimes I_{n_{i-1}}\otimes L_{n_r,\mu_r,\sigma_r}\otimes I_{n_{i+1}}\otimes\cdots\otimes I_{n_d},
\end{align}
where $\bmu=(\mu_1,\ldots,\mu_d)$ and $\bsigma=(\sigma_1,\ldots,\sigma_d)$. More details on the discretized multidimensional diffusion process considered here will be given in Section~\ref{sec:economics} along with an economics application; for more on diffusion processes, see \cite{Baldi} for a mathematical treatment and \cite{achdou2014partial,achdou2017income,gabaix2009power} for an economical application-oriented approach.

By the properties of tensor products and the results of Section~\ref{queue}, we can immediately obtain the eigendecomposition of $L_{\nn,\bmu,\bsigma}^\top$. In particular, the eigenpairs of $L_{\nn,\bmu,\bsigma}^\top$ are given by $(\nu_\kk,\ww_\kk)$, $\kk=\mathbf0,\ldots,\nn-\mathbf1$, where
\begin{equation}\label{L-eigenpairs}
\nu_\kk=\sum_{r=1}^d\nu_{k_r},\qquad\ww_\kk=\bigotimes_{r=1}^d\ww_{k_r},
\end{equation}
and $(\nu_{k_r},\ww_{k_r})$ is defined by \eqref{nu0w0}--\eqref{nukwk} for $(k,n,\lambda,\mu,\tau)=(k_r,n_r,\tilde\lambda_r,\tilde\mu_r,\tilde\tau_r)$ with $\tilde\tau_r=\sqrt{\tilde\lambda_r/\tilde\mu_r}$.

\begin{remark}[\textbf{Steady-State Distribution}]
The steady-state distribution of the $d$-dimensional diffusion process generated by $L_{\nn,\bmu,\bsigma}$, i.e., the normalized positive eigenvector of $L_{\nn,\bmu,\bsigma}^\top$ associated with the eigenvalue $0$, is given by
\begin{equation}\label{ssd}
\ppp = \frac{\ww_{\mathbf0}}{\|\ww_{\mathbf0}\|_1}=\bigotimes_{r=1}^d\frac{1-\tilde\rho_r}{1-\tilde\rho_r^{n_r}}\bigl[1,\tilde\rho_r,\tilde\rho_r^2,\ldots,\tilde\rho_r^{n_r-1}\bigr]^\top=\bigotimes_{r=1}^d\ppp_r,\qquad\tilde\rho_r=\tilde\tau_r^2,
\end{equation}
i.e., it is the tensor product of the steady-state distributions $\ppp_r$ of the individual unidimensional diffusion processes generated by the operators $L_{n_r,\mu_r,\sigma_r}$, $r=1,\ldots,d$.
\end{remark}

\subsection{Dynamics of Wealth and Income Inequality}\label{sec:economics}
In this section, we present an economic application of the results obtained in Section~\ref{mdp}. We begin with an overview of the topic, which may not be so familiar to non-economists.

\subsubsection{Modeling the Evolution of Wealth and Income}\label{sec:mod}
The sources of the vast wealth and income inequality is a key topic of study within macroeconomics and finance; see \cite{achdou2014partial,achdou2017income,atkinson2011top,benhabib2018skewed,benhabib2019wealth} for empirical evidence and modeling approaches. Central to the questions of inequality are:
\begin{itemize}[nolistsep,leftmargin=*]
	\item what is the source of heterogeneity that drives the stationary distribution of income or wealth?
	\item how would the income or wealth distribution evolve over time given aggregate changes?  
\end{itemize}
For example, researchers can ask how the stationary distribution of wealth will change---and how long it will take to be reached---given experiments such as a new income tax, technological changes driving more volatile wages, or increases in the returns on an asset such as housing.
Methodologically, the analysis of income inequality is done through examining the stationary distribution of discrete- or continuous-time stochastic processes associated with income or wealth. Typically, researchers act as follows.
\begin{itemize}[leftmargin=*,nolistsep]
	\item They choose a stochastic process for the assets of interest (for example, housing wealth, human wealth (i.e., wages), stocks, bonds, social security income, etc.). 
	\item They use data to estimate the parameters of the stochastic process for that ``portfolio'' of assets; see \cite{achdou2014partial} for a survey intended to bridge the continuous-time versions of these models. In some cases, the parameters are derived from optimal control of a Hamilton--Jacobi--Bellman equation \cite{achdou2014partial,achdou2017income,benhabib2019wealth}.
	\item They solve for the stationary distribution associated with the stochastic process. In this way, they can examine properties of the distribution, relate it back to the data, and conduct hypotheticals on the impact of policy.
\end{itemize}
With this approach, the emphasis on the steady-state distribution has come out of necessity. Even the speed of convergence towards the steady state has recently become an active research field; see \cite{gabaix2016dynamics} for a theory of the convergence rates largely focused on infinite-dimensional univariate models, and \cite{luttmer2007selection} for earlier evidence and theory on transition rates of the firm size distribution (methodologically, much of the literature on income/wealth inequality is similar to the firm dynamics literature, where the goal is to understand the distribution of firm sizes or productivity as well as the role of firm or worker heterogeneity in generating that distribution \cite{gabaix2009power,luttmer2007selection,luttmer2012slow}).

\subsubsection{Continuous-State Formulation}\label{csf}
Consider a portfolio of $d$ assets $\XX(t)=(X_1(t),\ldots,X_d(t))$ (e.g., housing wealth, wage income, social security, etc.). We emphasize the dependence on $t$ because the assets evolve over time. We assume that $X_1(t),\ldots,X_d(t)$ are $d$ independent Brownian motions with drifts $\mu_1,\ldots,\mu_d$ and variances $\sigma_1^2,\ldots,\sigma_d^2$.
Without loss of generality, we also assume that $X_1(t),\ldots,X_d(t)$ take values in $[0,1]$, so that the portfolio $\XX(t)$ determining an individual's wealth is an element $\xx=(x_1,\ldots,x_d)\in[0,1]^d$ at any time $t$.
The resulting stochastic process for the considered set of assets is a $d$-dimensional Brownian motion with drifts $\bmu=(\mu_1,\ldots,\mu_d)$ and variances $\bsigma^2=(\sigma_1^2,\ldots,\sigma_d^2)$, and with the edges of the hypercube $[0,1]^d$ acting as reflecting barriers.
The probability density function $p_r(x_r,t)$ for the asset $X_r(t)$ at time $t$ is determined by the Kolmogorov forward equation (Fokker--Planck equation)
\begin{align}
\frac{\partial p_r}{\partial t}(x_r,t) &= - \mu_r\frac{\partial p_r}{\partial x_r}(x_r,t) + \frac{\sigma_r^2}{2}\,\frac{\partial^2p_r}{\partial x_r^2}(x_r,t)\label{eq:pde1}
\end{align}
subject to the boundary conditions induced by reflecting boundaries at $0$ and $1$:
\begin{align}
0 &= -\mu_r p_r(x_r,t) + \frac{\sigma_r^2}{2}\,\frac{\partial p_r}{\partial x_r}(x_r,t),\qquad x_r=0,1.\label{eq:pde2}
\end{align}
The objects of interest are the following. 
\begin{itemize}[leftmargin=*,nolistsep]
	\item The stationary density function $p_r(x_r)$, that is, the density function independent of $t$ satisfying \eqref{eq:pde1}--\eqref{eq:pde2}. The function $p_r(x_r)$ does not evolve over time and determines the limiting (equilibrium) density function $p(\xx)=p_1(x_1)\cdots p_d(x_d)$ characterizing the steady-state probability distribution of the process.
	\item Any function $W$ that maps a state $\xx\in[0,1]^d$ to a scalar ``wealth'' or ``payoff'' $W(\xx)$.\,\footnote{\,As an example in the case $d=2$, asset $X_1(t)$ could be housing wealth at time $t$ and asset $X_2(t)$ could be bank holdings at time $t$ in an individual's portfolio. If $w_1$ is the per-unit value of a house and $w_2$ the per-unit value of a bank holding, then the ``wealth'' of an individual in state $(X_1,X_2)=(x_1,x_2)$ is $W(x_1,x_2)=w_1x_1+w_2x_2$.} Clearly, $W(\XX(t))$ is a random variable evolving over time together with the portfolio $\XX(t)$, and we are interested in quantities like the average wealth $\mathbb E[W(\XX)]$ and the wealth variance ${\rm Var}[W(\XX)]$ computed in the steady-state distribution $p(\xx)$, that is,
\begin{align*}
\mathbb E[W(\XX)]&=\int_{[0,1]^2}W(\xx)p(\xx){\rm d}\xx,\\
{\rm Var}[W(\XX)]&=\mathbb E[W(\XX)^2]-\mathbb E[W(\XX)]^2=\int_{[0,1]^2}W(\xx)^2p(\xx){\rm d}\xx-\biggl(\int_{[0,1]^2}W(\xx)p(\xx){\rm d}\xx\biggr)^2.
\end{align*}
\end{itemize}

\subsubsection{Discrete-State Formulation}\label{dsf}
Suppose we discretize the hypercube $[0,1]^d$ by introducing a $n_1\times\cdots\times n_d$ lattice with $n_r$ points in direction $x_r$ separated by a discretization step $\Delta_r>0$, as in Section~\ref{mdp}. This essentially means that we allow each random variable (asset) $X_r(t)$ to assume only a finite number of values.
Consequently, the portfolio $\XX(t)=(X_1(t),\ldots,X_r(t))$ can only be in a finite number of states $\xx_{\mathbf1},\ldots,\xx_\nn$.
The use of upwind finite differences allow us to convert the $2d$ PDEs \eqref{eq:pde1}--\eqref{eq:pde2} to a unique system of ODEs
\begin{align}
\frac{{\rm d}\ppp}{{\rm d}t}(t)&= L_{\nn,\bmu,\bsigma}^\top\mathbf{p}(t)\label{eq:p-ode}
\end{align}
subject to an initial condition $\ppp(0)$, where $L_{\nn,\bmu,\bsigma}$ is the infinitesimal generator \eqref{eq:diffusion-generator} and $\ppp_\ii(t)$ is the probability that the portfolio $\XX(t)$ is in state $\xx_\ii$ at time $t$. After this discretization, the continuous-state continuous-time Markov process of Section~\ref{csf} is changed into a discrete-state continuous-time Markov chain. Here, the objects of interest are the discrete counterparts of those mentioned in Section~\ref{csf}, i.e., the following.
\begin{itemize}[leftmargin=*,nolistsep]
	\item The stationary distribution $\ppp=(p_{\mathbf1},\ldots,p_\nn)$ of the process, that is, the probability vector independent of $t$ satisfying \eqref{eq:p-ode}. Clearly, $\ppp$ is the normalized positive eigenvector of $L_{\nn,\bmu,\bsigma}^\top$ associated with the zero eigenvalue and is given by \eqref{ssd}.
	\item Any function $W$ that maps a state $\xx_\ii\in[0,1]^d$ to a scalar ``wealth'' or ``payoff'' $W(\xx_\ii)=W_\ii$. Clearly, $W(\XX(t))$ is a random variable evolving over time together with the portfolio $\XX(t)$, and we are interested in quantities like the average wealth $\mathbb E[W(\XX)]$ and the wealth variance ${\rm Var}[W(\XX)]$ computed in the steady-state distribution $\ppp$, that is,
	\begin{align}
	\mathbb E[W(\XX)]&=\WW\cdot\ppp,\label{EW}\\
	{\rm Var}[W(\XX)]&=\mathbb E[W(\XX)^2]-\mathbb E[W(\XX)]^2=\WW^2\cdot\ppp-(\WW\cdot\ppp)^2,\label{VarW}
	\end{align}
	where $\WW=(W_{\mathbf1},\ldots,W_\nn)$ is the vector (tensor) of payoffs and $\WW^2$ is the componentwise square of $\WW$ (in general, operations on vectors that have no meaning in themselves must be interpreted in the componentwise sense).
\end{itemize}
Considering that $\ppp$ is known from \eqref{ssd}, formulas \eqref{EW}--\eqref{VarW} allow us to compute both the average wealth and the wealth variance in the steady state of the process. This lets us analyze different hypothetical scenarios. For example, if the drift $\mu_1$ of the housing component of an individual's portfolio increases, what would the impact be on the average wealth? Alternatively, we could ask how the wealth variance (a simple measure of inequality) would change if the variance of wages increases.

\subsubsection{Convergence Speed to the Steady State}
The results of Section~\ref{mdp} allow us to quantify the convergence speed to the steady state of the Markov chain presented in Section~\ref{dsf}.
Indeed, as we know from Section~\ref{mdp}, all nonzero eigenvalues of $L_{\nn,\bmu,\bsigma}$ are negative and the largest of them, i.e., the second largest eigenvalue after~0, is given by
\begin{equation}\label{2nd-eig}
\nu=\max_{r=1,\ldots,d}\Bigl(-\tilde\lambda_r-\tilde\mu_r+2\sqrt{\tilde\lambda_r\tilde\mu_r}\cos\frac\pi{n_r}\Bigr).
\end{equation}
The second eigenvalue provides a measure of the convergence speed towards the steady state. The reason is the following: for essentially every choice of the initial distribution $\ppp(0)$, the quantities $\ppp(t)$, $\mathbb E[W(\XX(t))]$, ${\rm Var}[W(\XX(t))]$ converge to their stationary counterparts $\ppp$, $\mathbb E[W(\XX)]$, ${\rm Var}[W(\XX)]$ in \eqref{ssd}, \eqref{EW}, \eqref{VarW} with asymptotic convergence rates given by
\begin{align}
\lim_{t\to\infty}\frac{{\rm d}}{{\rm d}t}\ln\|\ppp(t)-\ppp\|_2&=\nu,\label{p_rate}\\
\lim_{t\to\infty}\frac{{\rm d}}{{\rm d}t}\ln|\mathbb E[W(\XX(t))]-\mathbb E[W(\XX)]|&=\nu,\label{EW_rate}\\
\lim_{t\to\infty}\frac{{\rm d}}{{\rm d}t}\ln|{\rm Var}[W(\XX(t))]-{\rm Var}[W(\XX)]|&=\nu.\label{VarW_rate}
\end{align}
For more details on the role of the second eigenvalue as a measure of the asymptotic convergence rate towards the steady state, see, e.g., \cite{gabaix2009power} and \cite[Section~7.2]{lawler2006introduction}.

\subsubsection{Derivatives with Respect to Drifts and Variances}\label{sec:p-derivs}
For the convenience of economists, we here report the derivatives of the steady-state distribution $\ppp$ in \eqref{ssd}, the average wealth $\mathbb E[W(\XX)]$ in \eqref{EW}, and the wealth variance ${\rm Var}[W(\XX)]$ in \eqref{VarW} with respect to the drifts $\bmu$ and the variances $\bsigma^2$.
For $r=1,\ldots,d$, we have
{\allowdisplaybreaks\begin{align}
\frac{\partial\tilde\rho_r}{\partial\mu_r}&=\left\{\begin{aligned}&\frac{\sigma_r^2}{2\Delta_r^3\tilde\mu_r^2}=\frac{2\Delta_r\sigma_r^2}{(\sigma_r^2-2\Delta_r\mu_r)^2}, &&&\mbox{if $\mu_r\le0$},\\ &\frac{1}{\Delta_r\tilde\mu_r} = \frac{2\Delta_r}{\sigma_r^2}, &&&\mbox{if $\mu_r\ge0$},\end{aligned}\right.\notag\\[10pt]
\frac{\partial\tilde\rho_r}{\partial\sigma_r^2}&=-\frac{\mu_r}{2\Delta_r^3\tilde\mu_r^2},\notag\\[10pt]
\frac{\partial\ppp_r}{\partial\tilde\rho_r}&=\frac{(1-n_r)\tilde\rho_r^{n_r}+n_r\tilde\rho_r^{n_r-1}-1}{(1-\tilde\rho_r^{n_r})^2}\bigl[1,\tilde\rho_r,\tilde\rho_r^2,\ldots,\tilde\rho_r^{n_r-1}\bigr]^\top+\frac{1-\tilde\rho_r}{1-\tilde\rho_r^{n_r}}\bigl[0,1,2\tilde\rho_r,\ldots,(n_r-1)\tilde\rho_r^{n_r-2}\bigr]^\top,\notag\\[10pt]
\frac{\partial\ppp_r}{\partial\mu_r}&=\frac{\partial\tilde\rho_r}{\partial\mu_r}\,\frac{\partial\ppp_r}{\partial\tilde\rho_r},\notag\\[10pt]
\frac{\partial\ppp_r}{\partial\sigma_r^2}&=\frac{\partial\tilde\rho_r}{\partial\sigma_r^2}\,\frac{\partial\ppp_r}{\partial\tilde\rho_r},\notag\\[10pt]
\frac{\partial\ppp}{\partial\mu_r}&=\ppp_1\otimes\cdots\otimes\ppp_{r-1}\otimes\frac{\partial\ppp_r}{\partial\mu_r}\otimes\ppp_{r+1}\otimes\cdots\otimes\ppp_d,\label{dpmu}\\[10pt]
\frac{\partial\ppp}{\partial\sigma_r^2}&=\ppp_1\otimes\cdots\otimes\ppp_{r-1}\otimes\frac{\partial\ppp_r}{\partial\sigma_r^2}\otimes\ppp_{r+1}\otimes\cdots\otimes\ppp_d,\label{dpsigma}\\[10pt]
\frac{\partial\mathbb E[W(\XX)]}{\partial\mu_r}&=\WW\cdot\frac{\partial\ppp}{\partial\mu_r},\label{demu}\\[10pt]
\frac{\partial\mathbb E[W(\XX)]}{\partial\sigma_r^2}&=\WW\cdot\frac{\partial\ppp}{\partial\sigma_r^2},\label{desigma}\\[10pt]
\frac{\partial{\rm Var}[W(\XX)]}{\partial\mu_r}&=\WW^2\cdot\frac{\partial\ppp}{\partial\mu_r}-2(\WW\cdot\ppp)\Bigl(\WW\cdot\frac{\partial\ppp}{\partial\mu_r}\Bigr),\label{dvmu}\\[10pt]
\frac{\partial{\rm Var}[W(\XX)]}{\partial\sigma_r^2}&=\WW^2\cdot\frac{\partial\ppp}{\partial\sigma_r^2}-2(\WW\cdot\ppp)\Bigl(\WW\cdot\frac{\partial\ppp}{\partial\sigma_r^2}\Bigr).\label{dvsigma}
\end{align}}%
\begin{figure}
\centering
\includegraphics[width=\textwidth]{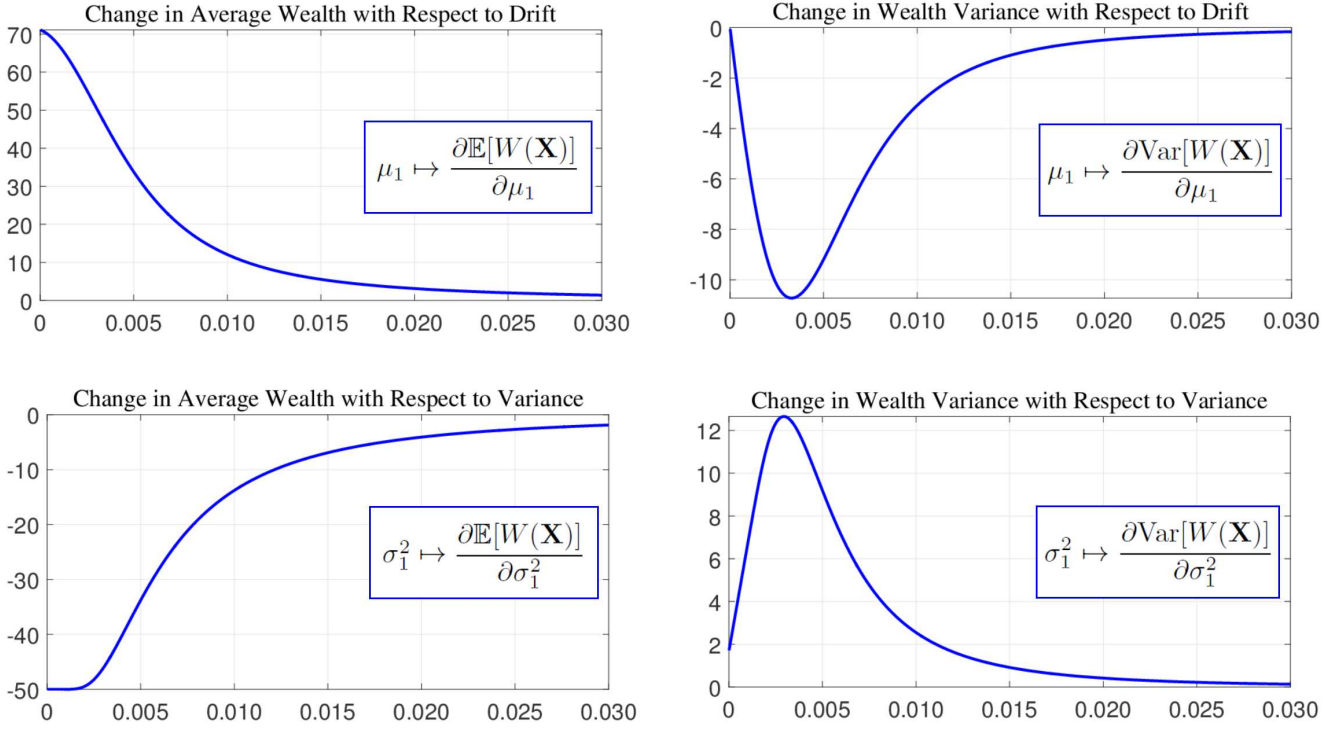}
\caption{Changes in moments of the wealth distribution. Parameters are chosen to be illustrative: $\mu_1 = \mu_2 = 0.01$, $\sigma_1^2 = \sigma_2^2 = 0.0025$, and $W(x_1,x_2)=x_1+x_2$.}
\label{fig:comparative-statics}
\end{figure}%
We remark that the above derivatives are defined even in the case $\mu_r=0$ and their values in this case are obtained by taking the limit of the corresponding expression as $\mu_r\to0$. The derivatives \eqref{dpmu}--\eqref{dpsigma} enable an analysis of how the steady state changes when properties of the underlying process change.  For example, if the volatility of housing prices $\sigma_1^2$ increases, equations~\eqref{dpmu}--\eqref{dpsigma} provide the resulting impact on the steady state.
The derivatives \eqref{demu}--\eqref{dvsigma} can be used to examine how key moments of the stationary distribution change. For example, a researcher could analyze the impact on the steady-state variance of the wealth distribution, i.e., ${\rm Var}[W(\XX)]$, in the case where the volatility of housing prices $\sigma_1^2$ increases.
Figure~\ref{fig:comparative-statics} illustrates this by showing how the mean and variance of the stationary wealth distribution change with respect to the parameters of the underlying stochastic process. The figure has been realized through a discretization of the square $[0,1]^2$ by a $n_1\times n_2$ lattice with $n_1=n_2=31$ points in each direction and (consequently) two equal discretization steps $\Delta_1=\Delta_2=1/30$. It should be noted, however, that the graphs in Figure~\ref{fig:comparative-statics} do not really depend on $n_1$ and $n_2$, because they converge to limiting graphs as $n_1,n_2\to\infty$ (and convergence is already reached for $n_1=n_2=31$).

\section{Conclusions and Perspectives}\label{sec:conc}
We have studied the spectral properties of the generator $T_{n,\epsilon,\phi}$ of the $\tau_{\epsilon,\phi}$ algebra introduced by Bozzo and Di Fiore in the context of matrix displacement decomposition \cite{BozzoDiFiore}. In particular:
\begin{itemize}[nolistsep,leftmargin=*]
	\item we have derived precise asymptotics for the outliers of $T_{n,\epsilon,\phi}$ and the associated eigenvectors;
	\item we have obtained equations for the eigenvalues of $T_{n,\epsilon,\phi}$, which automatically provide also the eigenvectors of $T_{n,\epsilon,\phi}$;
	\item we have computed the full eigendecomposition of $T_{n,\epsilon,\phi}$ in the case $\epsilon\phi=1$.
\end{itemize}
Finally, we have presented applications of our results to queuing models, random walks, diffusion processes, and economics, with a special emphasis on wealth/income inequality and portfolio dynamics.
We conclude this paper by mentioning a few possible future lines of research.
\begin{enumerate}[leftmargin=*,nolistsep]
	\item The applications presented herein do not exhaust all possible applications of the $\tau_{\epsilon,\phi}$ algebra. For example, matrices belonging to this algebra for suitable choices of $\epsilon$ and $\phi$ arise in the discretization of differential equations by finite difference methods, finite element methods and, as recently  discovered, isogeometric methods \cite[Section~3]{IgA-eig}. A future research could take care of investigating further discretizations where $\tau_{\epsilon,\phi}$ matrices arise and, consequently, the results of this paper find applications.
	\item On the economics side, Sections~\ref{mdp}--\ref{sec:economics} are interesting and useful, but the reflected constant-coefficient diffusion process $\XX(t)$ that has been considered therein is not sufficient to understand top income inequality, since in that case researchers need alternative specifications \cite{benhabib2018skewed,gabaix2016dynamics}. 
	That said, there could be a large class of stochastic processes $\hat\XX(t)$ that can be mapped to $\XX(t)$ through an appropriate change of measure.
	Loosely, given a stochastic process $\hat\XX(t)$, let $\hat W$ be a mapping such that $\hat W(\hat\XX(t))$ represents the ``wealth'' of an individual with portfolio $\hat\XX(t)$. Then, there may exist a change of measure $\mathbb P\to\mathbb Q$ (i.e., a Radon--Nikodym derivative ${\rm d}\mathbb Q/{\rm d}\mathbb P$) mapping $\hat\XX(t)$ to $\XX(t)$ and $\hat W(\hat\XX(t))$ to $W(\XX(t))$ for a suitable $W$. If so, then the computation of, say, the average wealth $\mathbb E_{\mathbb P}[\hat W(\hat\XX)]$ in the steady-state distribution of process $\hat\XX(t)$ could be traced back to computing the corresponding expectation $\mathbb E_{\mathbb Q}[W(\XX)]$ for process $\XX(t)$ as we have done in Section~\ref{sec:economics}; see \cite[Section~9.5]{campolieti2016financial} for an analysis of changes in probability measures and associated expectations, as well as for practical tools for working with such concepts.
	A careful investigation of all this topic may form the content of a future research that would extend the applicability of the results presented in this paper.
\end{enumerate}

\section*{Acknowledgements}
The authors wish to thank Carmine Di Fiore for useful discussions.
This work has been supported by the MIUR Excellence Department Project awarded to the Department of Mathematics of the University of Rome Tor Vergata (CUP E83C18000100006), by the Beyond Borders Programme of the University of Rome Tor Vergata through the Project ASTRID (CUP E84I19002250005), by the Research Group GNCS (Gruppo Nazionale per il Calcolo Scientifico) of INdAM (Istituto Nazionale di Alta Matematica), and by the Swedish Research Council through the International Postdoc Grant (Registration Number 2019-00495).

\end{document}